\documentclass[12pt]{amsart}
 
\usepackage{amssymb, amscd, txfonts}
\usepackage{graphicx}
\usepackage[all]{xy} 
\usepackage{array, booktabs}
\usepackage{tabularx} 

\numberwithin{equation}{section}

\newtheorem{theorem}{Theorem}[section]
\newtheorem{proposition}[theorem]{Proposition}
\newtheorem{lemma}[theorem]{Lemma}
\newtheorem{corollary}[theorem]{Corollary}

\theoremstyle{definition}

\newtheorem{example}[theorem]{Example}

\theoremstyle{remark}
\newtheorem{remark}[theorem]{Remark}

\renewcommand{\hom}{\operatorname{Hom}}

\newcommand{\Z}{\mathbb{Z}}
\newcommand{\Q}{\mathbb{Q}}
\newcommand{\R}{\mathbb{R}}
\newcommand{\C}{\mathbb{C}}
\newcommand{\proj}{{\mathbb P}}

\newcommand{\Fgn}{\mathcal{F}_{g,n}}

\newcommand{\Fg}{\mathcal{F}_{g}}

\newcommand{\Fgcpt}{\mathcal{F}_{g}^{\Sigma}}

\newcommand{\OL}{{\rm O}^{+}(\Lambda)}
\newcommand{\OLR}{{\rm O}^{+}(\Lambda_{\mathbb{R}})}
\newcommand{\OLg}{\tilde{{\rm O}}(\Lambda_{g})}
\newcommand{\G}{\Gamma}
\newcommand{\Gg}{\Gamma_{g}}

\newcommand{\LK}{\Lambda_{K3}}
\newcommand{\Lg}{\Lambda_{g}}
\newcommand{\OLK}{{\rm O}(\Lambda_{K3})}

\newcommand{\DK}{\Omega_{K3}}
\newcommand{\Dg}{\Omega_{g}}
\newcommand{\Dgo}{\Omega_{g}^{\circ}}
\newcommand{\BR}{\tilde{\Omega}_{K3}} 
\newcommand{\BRg}{\tilde{\Omega}_{g}}
\newcommand{\Xg}{\mathcal{X}_{g}}
\newcommand{\Xgcont}{\bar{\mathcal{X}}_{g}}
\newcommand{\Xgn}{\mathcal{X}_{g,n}}
\newcommand{\Xgo}{\mathcal{X}_{g}^{\circ}}
\newcommand{\XBR}{\tilde{\mathcal{X}}_{K3}}
\newcommand{\XgBR}{\tilde{\mathcal{X}}_{g}}

\newcommand{\D}{\mathcal{D}}
\newcommand{\volD}{{\rm vol}_{\mathcal{D}}}
\newcommand{\LI}{\Lambda(I)}
\newcommand{\LIQ}{\Lambda(I)_{\mathbb{Q}}}
\newcommand{\LIR}{\Lambda(I)_{\mathbb{R}}}
\newcommand{\LIC}{\Lambda(I)_{\mathbb{C}}}
\newcommand{\GIZ}{\Gamma(I)_{\mathbb{Z}}}
\newcommand{\GIQ}{\Gamma(I)_{\mathbb{Q}}}
\newcommand{\UIZ}{U(I)_{\mathbb{Z}}}
\newcommand{\UIQ}{U(I)_{\mathbb{Q}}}

\newcommand{\UIC}{U(I)_{\mathbb{C}}}

\begin{document}

\title[]{Mukai models and Borcherds products}
\author[]{Shouhei Ma}
\thanks{Supported by JSPS KAKENHI 15H05738 and 17K14158.} 
\address{Department~of~Mathematics, Tokyo~Institute~of~Technology, Tokyo 152-8551, Japan}
\email{ma@math.titech.ac.jp}
\subjclass[2010]{}

\begin{abstract}
Let ${\Fgn}$ be the moduli space of $n$-pointed $K3$ surfaces of genus $g$ 
with at worst rational double points. 
We establish an isomorphism between 
the ring of pluricanonical forms on ${\Fgn}$ 
and the ring of certain orthogonal modular forms, 
and give applications to the birational type of ${\Fgn}$. 
We prove that the Kodaira dimension of ${\Fgn}$ stabilizes to $19$ when $n$ is sufficiently large. 
Then we use explicit Borcherds products to find a lower bound of $n$ 
where ${\Fgn}$ has nonnegative Kodaira dimension, 
and compare this with an upper bound where ${\Fgn}$ is unirational or uniruled 
using Mukai models of $K3$ surfaces in $g\leq 20$. 
This reveals the exact transition point of Kodaira dimension in some $g$.  
\end{abstract} 

\maketitle

\section{Introduction}\label{sec:intro}

The moduli space of primitively polarized $K3$ surfaces of genus $g$ 
is identified with a Zariski open set 
of a $19$-dimensional modular variety ${\Fg}={\Gg}\backslash \mathcal{D}_{g}$ 
of orthogonal type, 
thanks to the Torelli theorem \cite{PS}, \cite{BR}. 
Its complement is the $(-2)$-Heegner divisor 
and parametrizes $K3$ surfaces having rational double points. 
Over ${\Fg}$ we have the moduli space ${\Fgn}$ of 
$n$-pointed $K3$ surfaces of genus $g$ with at worst rational double points (\S \ref{sec:2}). 
The purpose of this paper is to establish 
a correspondence between pluricanonical forms on ${\Fgn}$ 
and modular forms on ${\Fg}$, 
and give applications to the birational type of ${\Fgn}$. 
Specifically, 
we prove that the Kodaira dimension $\kappa({\Fgn})$ stabilizes to $19$ when $n$ is sufficiently large; 
and for $g\leq 20$ 
we study the transition point of $\kappa({\Fgn})$ from $-\infty$ to $\geq 0$
by combining the modular form method using explicit Borcherds products 
and the geometric method using Mukai models of polarized $K3$ surfaces. 
In the course we will also observe a curious coincidence \eqref{eqn:intro mystery}
between the Borcherds products and the Mukai models.  

Let $H^{0}(\mathcal{F}_{g,n}^{\circ}, K_{\mathcal{F}_{g,n}^{\circ}}^{\otimes m})$ be the space of 
$m$-canonical forms on the regular locus $\mathcal{F}_{g,n}^{\circ}$ of ${\Fgn}$. 
Let $M_{k}({\Gg}, \chi)$ be the space of modular forms on $\mathcal{D}_{g}$ 
of weight $k$ and character $\chi$ with respect to ${\Gg}$, 
$S_{k}({\Gg}, \chi)$ be the subspace of cusp forms, 
and $M_{k}({\Gg}, \chi)^{(m)}$ be the subspace of modular forms 
which vanish to order $\geq m$ at the $(-2)$-Heegner divisor. 
We have $M_{k}({\Gg}, \det)^{(1)}$=$M_{k}({\Gg}, \det)$ 
(Corollary \ref{cor: vanish at H}). 
Our point of departure is the following correspondence. 

\begin{theorem}[\S \ref{sec: theorem 1}]\label{main thm 1}
Let $(g, n)\ne (2, 1)$. 
We have an isomorphism of graded rings 
\begin{equation}\label{main eqn 1}
\bigoplus_{m} H^{0}(\mathcal{F}_{g,n}^{\circ}, K_{\mathcal{F}_{g,n}^{\circ}}^{\otimes m}) 
\: \simeq \: 
\bigoplus_{m} M_{(19+n)m}({\Gg}, {\det}^{m})^{(m)}. 
\end{equation}
If $X$ is a smooth projective model of ${\Fgn}$, 
this gives an isomorphism 
\begin{equation}\label{main eqn 2}
H^{0}(X, K_{X}) \simeq S_{19+n}({\Gg}, \det). 
\end{equation}
\end{theorem} 

If we consider only pluricanonical forms on 
the moduli space of \textit{smooth} $n$-pointed $K3$ surfaces, 
the modular forms can be meromorphic at the $(-2)$-Heegner divisor. 
Therefore, in order to fully understand the connection with modular forms, 
it is necessary to extend the universal family over the whole modular variety 
by allowing rational double points. 
(See also Remark \ref{remark: dispense with}.) 

As an application of \eqref{main eqn 2}, 
we study the birational type of ${\Fgn}$. 
By the Kodaira dimension $\kappa({\Fgn})$ of ${\Fgn}$, 
we mean the Kodaira dimension of its smooth projective model $X$. 
Then $\kappa({\Fgn})$ is nondecreasing with respect to $n$ (\cite{Ka}), 
and is bounded by $19$ (\cite{Ii}). 
The question whether $\kappa({\Fgn})$ indeed arrives at $19$ in large $n$ 
has remained (cf.~\cite{BM}). 
The interest is in the range $g\leq 62$, 
since $\kappa({\Fg})=19$ for $g \geq 63$ and some other $g$ 
by Gritsenko-Hulek-Sankaran \cite{GHS1}. 
The correspondence \eqref{main eqn 2} implies the following. 

\begin{corollary}\label{criteria for kappa}
Let $k_{0}$ be a weight such that $S_{k_{0}}({\Gg}, \det)\ne \{ 0 \}$, 
$k_{1}$ be a weight such that $\dim S_{k_{1}}({\Gg}, \det) > 1$, and 
$k_{2}$ be a weight such that $S_{k_{2}}({\Gg}, \det)$ 
gives a generically finite map 
${\Fg}\dashrightarrow {\proj}^{N}$. 
Then 

$(1)$ $\kappa({\Fgn})\geq 0$ for all $n\geq k_{0}-19$, 

$(2)$ $\kappa({\Fgn})>0$ for all $n\geq k_{1}-19$, and 

$(3)$ $\kappa({\Fgn})=19$ for all $n\geq k_{2}-19$. 

\noindent
In particular, $\kappa({\Fgn})$ stabilizes to $19$ for sufficiently large $n$. 
\end{corollary}

Here (3) holds because 
the canonical map of a smooth projective model $X$ factors through 
${\Fgn}\twoheadrightarrow {\Fg} \dashrightarrow {\proj}^{N}$ 
where ${\Fg} \dashrightarrow {\proj}^{N}$ is the rational map defined by 
$S_{19+n}({\Gg}, \det)$ (Corollary \ref{cor: canonical map}), 
and hence its image has dimension $19$ when $n=k_{2}-19$. 
The existence of such a weight $k_{2}$ is guaranteed by the general theory of Baily-Borel (\cite{BB}) 
and Hirzebruch-Mumford proportionality (\cite{GHS4}). 
 
We can always find, for every $g$, 
an explicit cusp form $F(g)$ of character $\det$ 
by quasi-pullback of the Borcherds $\Phi_{12}$ form 
(\cite{Bo}, \cite{BKPSB}, \cite{Ko}, \cite{GHS1}). 
If $k(g)$ is the weight of $F(g)$ and $n(g)=k(g)-19$, 
then $\mathcal{F}_{g,n(g)}$ has positive geometric genus, 
and $\kappa({\Fgn})\geq 0$ for $n\geq n(g)$. 
Usually the Borcherds product $F(g)$ has been used only when 
$k(g)\leq 19$ to study the birational type of ${\Fg}$ (\cite{Ko}, \cite{GHS1}). 
But even when $k(g)>19$, it can be thus used to study the universal families over ${\Fg}$. 

On the other hand, for most $g\leq 22$, 
general $K3$ surfaces of genus $g$ have been explicitly studied by 
Mukai \cite{Mu1} -- \cite{Mu6}, Farkas-Verra \cite{FV}, \cite{FV2} and Barros \cite{Ba}. 
This tells us a bound $n'(g)$  
where ${\Fgn}$ is unirational or uniruled. 
We compare the arithmetic bound $n(g)$ for $\kappa({\Fgn})\geq0$ 
with the geometric bound $n'(g)$ for $\kappa({\Fgn})=-\infty$.

\begin{theorem}[\S \ref{sec: Mukai Borcherds}]\label{main thm Mukai Borcherds}
For $g \leq 22$,  
$n(g)$ and $n'(g)$ are as in the following table. 
In particular, $n(g)=n'(g)+1$ for $g=3, 4, 6, 12, 20$. 

\begin{center}
\begin{tabular}{cccccccccccc} \toprule  
$g$     & 2   & 3   & 4   & 5   & 6   & 7 & 8 & 9 & 10 & 11 & 12   
\\ \midrule  
$n(g)$ & 56 & 35 & 30 & 21 & 23 & 16 & 15 & 14 & 14 & 9 & 14  
\\ \midrule  
$n'(g)$ & 38 & 34 & 29 & 18 & 22 & 14 & 9 & 10 & 11 & 7 & 13  
\\ \bottomrule 
\end{tabular}
\end{center}
\begin{center}
\begin{tabular}{ccccccccccc} \toprule  
$g$    &  13 & 14 & 15 & 16 & 17 & 18 & 19 & 20 & 21 & 22 \\ 
\midrule  
$n(g)$ &  9 & 9 & 8 & 7 & 6 & 8 & 5 & 6 & 5 & 6 \\ 
\midrule  
$n'(g)$ &  7 & 1 &   & 4 &    & 5 &    & 5 &    & 1 \\ 
\bottomrule 
\end{tabular}
\end{center}
\end{theorem}

Here $\mathcal{F}_{g,n'(g)}$ is uniruled for $g= 7, 11, 13, 16, 18, 20$, 
rationally connected for $g=12$, 
and (uni)rational for other $g$. 
The bound $n'(g)$ in $8 \leq g \leq 10$ is due to Farkas-Verra \cite{FV} using the Mukai models in \cite{Mu1}, 
and in $g=11, 14, 22$ due to Barros \cite{Ba} and Farkas-Verra \cite{FV}, \cite{FV2} respectively. 
For other $g$ we compute $n'(g)$ 
by a geometric argument using classical models ($g \leq 5$) 
and Mukai models \cite{Mu1} -- \cite{Mu6} ($g\geq 6$). 
When $g=7$, we improve the bound of \cite{FV}.  
The existence of a canonical form on $\mathcal{F}_{11,9}$  
matches nicely with the result 
$\kappa(\mathcal{F}_{11,9})=0$ of Barros-Mullane \cite{BM}. 
In $g=7, 13, 16$, we also find that 
a space akin to $\mathcal{F}_{g,n'(g)+1}$ is uniruled (dual $K3$ fibration). 

One observes that $n(g)$ is close to $n'(g)$ in relatively many cases. 
Thus sandwich by Mukai models and Borcherds products, 
two techniques of different nature,  
tells us a rather precise information 
on the transition of the birational type of ${\Fgn}$. 
This is what the title means. 
Moreover, in every $3\leq g\leq 10$, we find (\S \ref{mystery}) that the coincidence 
\begin{equation}\label{eqn:intro mystery}
n(g) = \dim V_{g} 
\end{equation}
holds, 
where $V_{g}$ is a representation of an algebraic group appearing in the classical/Mukai models. 
This might suggest a further link between 
the Borcherds product $F(g)$ and the Mukai model. 

A related subject is the Kodaira dimension of the moduli space $\mathcal{M}_{g,n}$ of $n$-pointed curves of genus $g$, 
which has been studied by Logan \cite{Lo}, Farkas \cite{Fa} and Agostini-Barros \cite{AB} 
(and for $n=0$, Harris-Mumford-Eisenbud \cite{HM}, \cite{EH} and Farkas-Jensen-Payne \cite{FJP}). 
Incidentally, in many cases in $7\leq g \leq 22$, $n(g)$ is near to 
the Logan-Farkas bound for $\kappa(\mathcal{M}_{g,n})\geq 0$.

As for the criterion (2) in Corollary \ref{criteria for kappa}, 
we can find a weight $k_{1}$ such that $\dim S_{k_{1}}({\Gg}, \det)>1$ 
by multiplying $F(g)$ and a space $M_{k}({\Gg})=M_{k}({\Gg}, 1)$ with $\dim M_{k}({\Gg})>1$. 
Such a space $M_{k}({\Gg})$ can be explicitly found, e.g., by the Jacobi lifting \cite{Gr}. 
Similarly, for the criterion (3), 
we can find a weight $k_{2}$ there by multiplying $F(g)$ and 
a space $M_{k}({\Gg})$ that gives a generically finite map ${\Fg}\dashrightarrow {\proj}^{N}$. 
Although we know that such a weight $k$ exists by the Baily-Borel theory, 
it is in general not easy to explicitly calculate it.  
Moreover,  the resulting bound $n(g)+k$ for $\kappa({\Fgn})=19$ 
would be far from the actual bound, 
which is expected to be near to $n(g)$. 

These points could be improved 
if ${\Fgn}$ admits a compactification $X$ 
which is an equidimensional family over some toroidal compactification of ${\Fg}$ in codimension $1$  
(\S \ref{ssec: main thm 2}). 
For such a compactification, 
the boundary obstruction for pluricanonical forms can be estimated in terms of modular forms. 
Consequently, \eqref{main eqn 1} extends to an isomorphism 
with the log canonical ring of ${\Fgn}\hookrightarrow X$, 
and moreover, 
assuming $X$ has canonical singularities, 
we have $\kappa({\Fgn})=19$ whenever $n>n(g)$.  
However, unlike the case of abelian varieties, 
no example of such a compactification has been known, 
which makes this part conditional. 
The singularities of ${\Fgn}$ 
can be studied through the theory of semi-universal deformation of rational double points 
and the theory of automorphisms of $K3$ surfaces. 
Anyway, Corollary \ref{criteria for kappa} enables  
to study $\kappa({\Fgn})$ without knowing explicit compactification. 

The origin of this paper goes back to the work of Shioda \cite{Sh} 
on the correspondence between 
elliptic cusp forms of weight $3$ and canonical forms on the elliptic modular surfaces. 
This was generalized by Shokurov \cite{Shoku} to $n$-pointed elliptic curves, 
and by Hatada \cite{Ha} 
to $n$-pointed abelian varieties and Siegel modular forms. 
Another generalization of elliptic curves are $K3$ surfaces. 
The new feature here is the degenerate family over the $(-2)$-Heegner divisor, 
which is responsible for the vanishing condition there. 

This paper is organized as follows. 
In \S \ref{sec:2} we give an analytic construction of ${\Fgn}$.  
In \S \ref{sec: modular form} we recall the basic theory of orthogonal modular forms. 
In \S \ref{sec: theorem 1} we prove Theorem \ref{main thm 1}. 
In \S \ref{sec: Mukai Borcherds} we prove Theorem \ref{main thm Mukai Borcherds}.  

I would like to thank Igor Dolgachev for his help in genus $3$,  
and Gavril Farkas, Shigeru Mukai and Alessandro Verra for their valuable comments. 
I would also like to thank the referee for many helpful comments.

\section{Universal $K3$ surface over the period domain}\label{sec:2}

In this section we give an analytic construction of ${\Fgn}$. 
This is necessary for the connection with modular forms. 
The effort is paid for equivariant extension over the $(-2)$-Heegner divisor. 
The main results are Propositions \ref{prop:extend family} and \ref{prop: ramification divisor}. 

Let $g\geq 2$. 
Let ${\LK}=3U\oplus 2E_{8}$ be the $K3$ lattice. 
We fix a primitive vector $l\in{\LK}$ of norm $2g-2$, 
which is unique up to the ${\OLK}$-action (\cite{Ni1}). 
The polarized $K3$ lattice of degree $2g-2$ is defined as its orthogonal complement  
\begin{equation*}
{\Lg} = l^{\perp}\cap {\LK} 
\simeq 2U \oplus 2E_{8} \oplus \langle 2-2g \rangle. 
\end{equation*} 
The polarized period domain is defined by 
\begin{equation*}
{\Dg} = 
\{ \: [\omega ] \in {\proj}({\Lg})_{{\C}} \: | \: 
(\omega, \omega)=0, \: (\omega, \bar{\omega})>0 \: \}. 
\end{equation*}
This consists of two connected components, 
each of which is a Hermitian symmetric domain of type IV. 

Let ${\OLg}$ be the kernel of the reduction map 
${\rm O}({\Lg}) \to {\rm O}(\Lambda_{g}^{\vee}/\Lambda_{g})$. 
By Nikulin \cite{Ni1}, ${\OLg}$ is identified with the stabilizer of $l$ in ${\rm O}(\Lambda_{K3})$. 
We set 
\begin{equation*}
{\Fg} = {\OLg} \backslash {\Dg}. 
\end{equation*}
Since ${\OLg}$ has an element exchanging the two components of ${\Dg}$, 
${\Fg}$ is irreducible. 
By Baily-Borel \cite{BB}, ${\Fg}$ is a normal quasi-projective variety.  

The $(-2)$-Heegner divisor of ${\Dg}$ is defined as 
\begin{equation}\label{eqn: (-2)-Heegner K3}
\mathcal{H} = 
\bigcup_{\begin{subarray}{c} \delta\in {\Lg} \\ (\delta, \delta)=-2 \end{subarray}} 
\delta^{\perp}\cap{\Dg}. 
\end{equation}
This is locally finite and descends to an algebraic divisor of ${\Fg}$. 
We write 
${\Dgo} = {\Dg} - \mathcal{H}$. 

A \textit{marked polarized $K3$ surface} of genus $g$ 
is a triplet $(S, L, \varphi)$ where 
$S$ is a smooth $K3$ surface, 
$L$ is an ample line bundle on $S$ of degree $2g-2$, and 
$\varphi$ is an isometry $H^{2}(S, {\Z})\to {\LK}$ such that $\varphi([L])=l$. 
(In particular, $L$ is assumed to be primitive.) 
When $L$ is nef and big, we say instead \textit{quasi-polarized}. 
By the period mapping, ${\Dgo}$ is identified with 
the fine moduli space of marked polarized $K3$ surfaces of genus $g$. 
Gluing the polarized Kuranishi families, 
we have a universal family 
$({\Xgo} \to {\Dgo}, \varphi)$ over ${\Dgo}$ (see \cite{Be}, \cite{Hu}).  
We have the equivariant action of ${\OLg}$ on ${\Xgo} \to {\Dgo}$ 
induced by changing the marking (cf.~\cite{Hu} p.120). 

Our purpose in this section is to extend ${\Xgo} \to {\Dgo}$ 
to a family over ${\Dg}$ together with the action of ${\OLg}$, 
and determine its fixed divisor. 
The results proved in \S \ref{ssec: extend family} and  \S \ref{ssec: OLg-action}, 
after preliminaries in \S \ref{ssec:BR domain}, are summarized as follows. 

\begin{proposition}\label{prop:extend family}
Let $g\geq 2$.  

(1) The family ${\Xgo} \to {\Dgo}$ extends to 
a projective flat family 
$\pi\colon {\Xg} \to {\Dg}$ over ${\Dg}$ 
of polarized $K3$ surfaces with at worst rational double points, 
and the ${\OLg}$-action on ${\Xgo}$ extends to a ${\OLg}$-action on ${\Xg}$. 

(2) The reflection with respect to a $(-2)$-vector $\delta\in {\Lg}$ 
acts trivially on the total space $\pi^{-1}(\delta^{\perp}\cap{\Dg})$ 
over $\delta^{\perp}\cap{\Dg}$. 
When $g\geq3$, every fixed divisor of the ${\OLg}$-action on ${\Xg}$ is of this form. 
When $g=2$, we have an additional fixed divisor formed by 
the ramification curves of the double planes. 
\end{proposition}

In \S \ref{ssec: Fgn} we take the $n$-fold fiber product ${\Xgn}$ of ${\Xg}\to{\Dg}$ 
and the quotient ${\Fgn}={\OLg}\backslash{\Xgn}$. 
This is an analytic construction of the moduli space 
of $n$-pointed $K3$ surfaces of genus $g$ with at worst rational double points, 
and will be a basis of the connection with modular forms.

\subsection{Polarized Burns-Rapoport period domain}\label{ssec:BR domain}

We recall the Burns-Rapoport period domain 
(\cite{BR}, \cite{Be})  
and consider its polarized version. 

\subsubsection{Burns-Rapoport period domain}\label{sssec:BR domain} 

Let 
\begin{equation*}
{\DK} = 
\{ \: [\omega ] \in {\proj}({\LK})_{{\C}} \: | \: 
(\omega, \omega)=0, \: (\omega, \bar{\omega})>0 \: \} 
\end{equation*}
be the period domain of $K3$ surfaces.   
The period of a marked $K3$ surface $(S, \varphi)$ is defined by 
$\mathcal{P}(S, \varphi) = \varphi(H^{2,0}(S)) \in {\DK}$.  
For $[\omega]\in {\DK}$ we write 
$NS(\omega)=\omega^{\perp}\cap {\LK}$,  
$H^{1,1}_{{\R}}(\omega)=\omega^{\perp}\cap ({\LK})_{{\R}}$ and 
$\mathcal{V}(\omega)= \{ x\in H^{1,1}_{{\R}}(\omega) | (x, x)>0 \}$. 
Let $\Delta(\omega)$ be the set of $(-2)$-vectors in $NS(\omega)$ 
and $W(\omega)$ be its Weyl group. 
Connected components of 
$\mathcal{V}(\omega)- \cup_{\delta\in \Delta(\omega)}\delta^{\perp}$ 
are called Weyl chambers of $\mathcal{V}(\omega)$. 
The group $W(\omega)\times \{ \pm {\rm id} \}$ acts freely and transitively 
on the set of Weyl chambers of $\mathcal{V}(\omega)$. 
The Burns-Rapoport period domain (\cite{BR}) is set-theoretically defined as 
\begin{equation*}
{\BR} = \{ \: ([\omega], \mathcal{C}) \: | \: [\omega]\in{\DK}, \; 
\mathcal{C} \: \textrm{a Weyl chamber of} \: \mathcal{V}(\omega) \: \}. 
\end{equation*}
If $(S, \varphi)$ is a marked $K3$ surface, 
its Burns-Rapoport period is defined by 
\begin{equation*}
\mathcal{P}_{BR}(S, \varphi) = 
(\varphi(H^{2,0}(S)), \varphi(\mathcal{K}_{S})) \in {\BR}, 
\end{equation*}
where $\mathcal{K}_{S}$ is the K\"ahler chamber of $S$. 
By the strong Torelli theorem, 
this identifies ${\BR}$ with the fine moduli space of marked $K3$ surfaces (\cite{BR}, \cite{Be}). 
The marked Kuranishi families equip ${\BR}$ with 
the structure of a non-Haussdorff complex manifold 
and a universal family 
$(\tilde{\mathcal{X}}_{K3} \stackrel{\pi}{\to}{\BR}, \varphi)$, 
where $\varphi$ is an isomorphism 
$R^2\pi_{\ast}{\Z}\to \underline{{\LK}}$ of local systems. 
The projection ${\BR}\to{\DK}$ is locally isomorphic by the local Torelli theorem.

The group ${\OLK}$ acts on 
$\tilde{\mathcal{X}}_{K3}\to {\BR}$ equivariantly as follows. 
Let $\gamma\in {\OLK}$. 
If $(S, \varphi)$ is a marked $K3$ surface, 
$\gamma$ sends $[(S, \varphi)]\in{\BR}$ to 
$[(S, \gamma\circ\varphi)]\in{\BR}$. 
The isomorphism between the fibers over them 
is ``the identity map'' $S\to S$ of $S$. 
To be more intrinsic, 
if $(S', \varphi')$ is a marked $K3$ surface 
isomorphic to $(S, \gamma\circ\varphi)$, 
the fiber map is the isomorphism $f:S\to S'$ with 
$\varphi' \circ f_{\ast} = \gamma \circ \varphi$ 
(which uniquely exists by the strong Torelli theorem). 
This defines the action of ${\OLK}$ on $\tilde{\mathcal{X}}_{K3}$.

\subsubsection{Polarized Burns-Rapoport period domain}

Next we consider the polarized version.  
Let ${\BRg}\subset{\BR}$ be the subset consisting of those $([\omega], \mathcal{C})$ 
such that $[\omega]\in{\Dg}$ and that the Weyl chamber $\mathcal{C}$ contains $l$ in its closure. 
By the Burns-Rapoport period mapping, 
${\BRg}$ is identified with the set of isomorphism classes of 
marked quasi-polarized $K3$ surfaces of genus $g$. 
The basic observation is 

\begin{proposition}\label{prop: pol BR domain}
${\BRg}$ is a complex submanifold of ${\BR}$. 
\end{proposition}

This is a consequence of the following general property. 

\begin{lemma}\label{lem: nef big stable}
Let $X\to U$ be a family of $K3$ surfaces with a line bundle $L$ on $X$. 
Let $L_{u}=L|_{X_{u}}$ for $u\in U$. 
Assume that $L_{u_{0}}$ is nef and big for a point $u_{0}\in U$. 
Then $L_{u}$ is nef and big for all $u\in U$ in an open neighborhood of $u_{0}$. 
\end{lemma}

Indeed, for a marked Kuranishi family 
$(\mathcal{X}\to\mathcal{U}, \varphi)$ 
such that 
$\mathcal{P}_{BR}(\mathcal{X}_{u_{0}}, \varphi_{u_{0}})\in {\BRg}$, 
Lemma \ref{lem: nef big stable} implies that 
$\varphi_{u}^{-1}(l)\in NS(\mathcal{X}_{u})$ is nef and big 
for all $u\in \mathcal{P}^{-1}(\Omega_{g})$ 
in a neighborhood of $u_{0}$. 
Hence, shrinking $\mathcal{U}$ if necessary, we have 
$\mathcal{P}_{BR}^{-1}({\BRg})=\mathcal{P}^{-1}({\Dg})$ 
in $\mathcal{U}$. 
Since $\mathcal{P}\colon \mathcal{U}\to {\DK}$ is an open immersion 
and $\Omega_{g}$ is a complex submanifold of ${\DK}$, 
this shows that 
$\mathcal{P}_{BR}^{-1}({\BRg})$ 
is a complex submanifold of $\mathcal{U}$. 
Recalling that the complex structure of ${\BR}$ is induced from the Kuranishi families, 
this proves Proposition \ref{prop: pol BR domain}. 

Lemma \ref{lem: nef big stable} follows from   
the following cohomological characterization of nef and big line bundle and 
the upper semicontinuity of $h^{i}(mL_{u})$. 
 
\begin{proposition}\label{prop: h1=h2=0}
Let $S$ be a $K3$ surface and $L$ be a line bundle with $(L, L)>0$. 
Then $L$ is nef and big if and only if 
$h^{1}(mL)=h^{2}(mL)=0$ for some $m\geq 2$. 
\end{proposition}

\begin{proof}
The ``only if'' part is the Kodaira-Ramanujam vanishing theorem for $mL$. 
We prove the ``if'' part. 
Suppose $h^{1}(mL)=h^{2}(mL)=0$. 
Since $h^{0}(mL)>0$ by Riemann-Roch, 
$L$ is in the positive cone of $NS(S)_{{\R}}$. 
Assume to the contrary that $L$ is not nef and big. 
Then $L$ is not nef, i.e., not in the closure of the ample cone, 
because a nef bundle $L$ with $(L, L)>0$ is also big. 
Take a Weyl chamber $\mathcal{C}$ of $NS(S)_{{\R}}$ with $L\in \overline{\mathcal{C}}$. 
There exists a sequence of effective $(-2)$-vectors 
$\delta_{1}, \cdots, \delta_{N} \in NS(S)$ 
such that, if $s_{i}$ is the reflection by $\delta_{i}$ and 
$\mathcal{C}_{i}=s_{i}(\mathcal{C}_{i-1})$ with $\mathcal{C}_{0}=\mathcal{C}$, 
then $\mathcal{C}_{N}$ is the ample cone and 
$(\mathcal{C}_{i-1}, \delta_{i})<0$ for all $i$. 
(Connect $\mathcal{C}$ and the ample cone by a general segment.) 
We set $L_{0}=L$ and $L_{i}=s_{i}(L_{i-1})$. 
Then $L_{i}\in \overline{\mathcal{C}}_{i}$. 

We claim that 
$h^{0}(mL_{N})<h^{0}(mL)$. 
Indeed, since $L_{i}=L_{i-1}+(\delta_{i}, L_{i-1})\delta_{i}$ and 
$(\delta_{i}, L_{i-1})\leq 0$, 
we have $h^{0}(mL_{i})\leq h^{0}(mL_{i-1})$ for every $i$. 
We look at $i=N$. 
Note that $(L_{N}, \delta_{k})\geq 0$ for every $k$ because $L_{N}$ is nef. 
Let $j$ be the largest index such that $(L_{N}, \delta_{j})>0$, 
which exists because $L\ne L_{N}$. 
There exists an irreducible curve $C\leq \delta_{j}$ with $(L_{N}, C)>0$. 
By Saint-Donat's result (\cite{SD} \S 2 -- \S 3, see also \cite{Hu} Chapter 2.3), 
we may assume that $C$ is smooth. 
Consider the exact sequence 
\begin{equation*}
0 \to H^{0}(mL_{N}) \to H^{0}(mL_{N}+C) \to H^{0}((mL_{N}+C)|_{C}) \to H^{1}(mL_{N}). 
\end{equation*}
We have $h^{1}(mL_{N})=0$ because $L_{N}$ is nef and big, 
and $h^{0}((mL_{N}+C)|_{C})>0$ because 
${\rm deg}((mL_{N}+C)|_{C})\geq 2g_{C}$ by $m\geq 2$. 
It follows that 
\begin{equation*}
h^{0}(mL_{N}) 
< h^{0}(mL_{N}+C) 
\leq h^{0}(mL_{N}+\delta_{j}) 
= h^{0}(mL_{j}+\delta_{j}) 
\leq h^{0}(mL_{j-1}). 
\end{equation*}
This proves $h^{0}(mL_{N})<h^{0}(mL)$. 

On the other hand, since $(L_{N}, L_{N})=(L, L)$, 
this contradicts the consequence of the Riemann-Roch formula 
\begin{equation*}
h^{0}(mL_{N}) 
= m^{2}(L_{N}, L_{N})/2+2 
= m^{2}(L, L)/2+2 
= h^{0}(mL). 
\end{equation*}
This proves Proposition \ref{prop: h1=h2=0} and 
so finishes the proof of Lemma \ref{lem: nef big stable} and Proposition \ref{prop: pol BR domain}. 
\end{proof}

The fibers of the projection 
$p\colon {\BRg}\to {\Dg}$ are described as follows. 
For $[\omega]\in{\Dg}$ we let 
$NS(\omega, l)=l^{\perp}\cap NS(\omega)$ and 
$\Delta(\omega, l)$ be the set of $(-2)$-vectors in $NS(\omega, l)$. 
Since $NS(\omega, l)$ is negative-definite, 
the root lattice $R(\omega, l)\subset NS(\omega, l)$  
generated by $\Delta(\omega, l)$ is an orthogonal sum of some ADE root lattices. 
Let $W(\omega, l)$ be its Weyl group. 

\begin{lemma}\label{lem:fiber of pol BR}
The fiber $p^{-1}([\omega])$ is identified with 
the set of Weyl chambers of $R(\omega, l)_{{\R}}$. 
In particular, $W(\omega, l)$ acts on $p^{-1}([\omega])$ freely and transitively. 
\end{lemma}

\begin{proof}
The fiber $p^{-1}([\omega])$ is the set of 
Weyl chambers of $\mathcal{V}(\omega)$ 
which contains $l$ in its closure. 
If $\mathcal{V}^{+}(\omega)$ is the component of 
$\mathcal{V}(\omega)$ containing $l$,  
every connected component of 
$\mathcal{V}^{+}(\omega)- \cup_{\delta\in \Delta(\omega, l)}\delta^{\perp}$  
contains a unique such Weyl chamber.  
Those components in turn correspond bijectively with the Weyl chambers of $R(\omega, l)_{{\R}}$. 
\end{proof}

This shows that 
${\Fg}\simeq {\OLg}\backslash {\BRg}$, 
so ${\Fg}$ is the moduli space of quasi-polarized $K3$ surfaces of genus $g$. 
Equivalently, contracting the $(-2)$-curves orthogonal to the quasi-polarization, 
${\Fg}$ is also the moduli space of 
polarized $K3$ surfaces of genus $g$ with at worst rational double points.

\subsubsection{Universal family over ${\BRg}$}\label{sssec: univ family over BRg}

Let $({\XBR}\to{\BR}, \varphi)$ be the universal marked family  
over ${\BR}$ as in \S \ref{sssec:BR domain}. 
By restriction to the submanifold ${\BRg}$, 
we obtain a marked family 
$({\XgBR}\stackrel{\pi}{\to}{\BRg}, \varphi)$ over ${\BRg}$. 
If we take a sufficiently fine open covering 
$(\mathcal{U}_{\alpha})_{\alpha}$ 
of ${\BRg}$ and put 
$\mathcal{X}_{\alpha}=\pi^{-1}(\mathcal{U}_{\alpha})$, 
there exists a collection 
$\mathcal{L}=(\mathcal{L}_{\alpha})_{\alpha}$ 
of line bundles, 
each $\mathcal{L}_{\alpha}$ on $\mathcal{X}_{\alpha}$, such that 
$[(\mathcal{L}_{\alpha})_{u}]=\varphi_{u}^{-1}(l)$ for all $u\in \mathcal{U}_{\alpha}$ 
(see, e.g., \cite{Hu} p.110). 
We will call 
$(\mathcal{X}_{\alpha}\to \mathcal{U}_{\alpha}, \mathcal{L}_{\alpha}, \varphi)$ 
a (marked quasi-polarized) Kuranishi family. 
Over $\mathcal{X}_{\alpha}\cap \mathcal{X}_{\beta}$ 
there exists an isomorphism 
$\mathcal{L}_{\alpha}|_{\mathcal{X}_{\alpha}\cap \mathcal{X}_{\beta}} \to 
\mathcal{L}_{\beta}|_{\mathcal{X}_{\alpha}\cap \mathcal{X}_{\beta}}$ 
unique up to 
$\mathcal{O}^{\ast}(\mathcal{U}_{\alpha}\cap \mathcal{U}_{\beta})$. 
In particular, for every $k$, 
the local ${\proj}^{N}$-bundles 
${\proj}(\pi_{\ast}\mathcal{L}_{\alpha}^{\otimes k})^{\vee}$ 
are canonically glued to a ${\proj}^{N}$-bundle over ${\BRg}$. 
By abuse of notation, we denote this by 
${\proj}(\pi_{\ast}\mathcal{L}^{\otimes k})^{\vee}$. 
(As we will use only the projective morphisms 
$\mathcal{X}_{\alpha}\to {\proj}(\pi_{\ast}\mathcal{L}_{\alpha}^{\otimes 3})^{\vee}$, 
the local collection $(\mathcal{L}_{\alpha})$ is sufficient for our purpose.)

Since the ${\OLg}$-action on ${\BR}$ preserves ${\BRg}$, 
the equivariant action on ${\XBR}$ preserves ${\XgBR}$. 
Thus ${\OLg}$ acts on ${\XgBR}\to{\BRg}$ equivariantly. 
If $\gamma\in{\OLg}$ and $[(S, \varphi)]\in{\BRg}$, 
then $\gamma$ sends $[(S, \varphi)]$ to $[(S, \gamma\circ\varphi)]$, 
and the isomorphism between the fibers over them is given by 
``the identity'' $S\to S$ of $S$.

\subsection{Extension of the universal family}\label{ssec: extend family}

We now construct a family ${\Xg}\to{\Dg}$  
as in Proposition \ref{prop:extend family} (1). 
We do this in two steps: 
first contract ${\XgBR}$ over ${\BRg}$, 
and then show that this descends to a family over ${\Dg}$. 

Let $({\XgBR}\stackrel{\pi}{\to}{\BRg}, \mathcal{L}, \varphi)$ be the universal family 
as in \S \ref{sssec: univ family over BRg}. 
Since $\mathcal{L}_{u}^{\otimes 3}$ is base point free for every $u\in{\BRg}$ 
(\cite{SD} Theorem 8.3, see also \cite{Hu} Chapter 2 Remark 3.4), 
we can take the relative projective morphism 
${\XgBR}\to {\proj}(\pi_{\ast}\mathcal{L}^{\otimes 3})^{\vee}$ 
over ${\BRg}$. 
Let ${\Xgcont}$ be the image of this morphism. 
This is a flat projective family of $K3$ surfaces with at worst rational double points. 
By \cite{EGA4} Corollaire 6.5.4, ${\Xgcont}$ is normal. 
At each fiber, $({\XgBR})_{u}\to ({\Xgcont})_{u}$ contracts 
all $(-2)$-curves orthogonal to $\mathcal{L}_{u}$. 

\begin{proposition}\label{prop:descend family}
There exists a family ${\Xg}\to{\Dg}$ of $K3$ surfaces with at worst rational double points 
such that ${\Xgcont}$ is isomorphic to the pullback of ${\Xg}$ to ${\BRg}$. 
\end{proposition}

\begin{proof}
Let  
$(\mathcal{X}\to \mathcal{U}, \mathcal{L}, \varphi)$ 
be a marked quasi-polarized Kuranishi family of 
$(S, L, \varphi)=(\mathcal{X}_{u_{0}}, \mathcal{L}_{u_{0}}, \varphi_{u_{0}})$. 
We set $[\omega]=\mathcal{P}(u_{0})$ and  
$\mathcal{U}^{\circ} = \mathcal{U} - \mathcal{P}^{-1}(\mathcal{H})$. 
We may assume that $[\omega]\in\mathcal{H}$.  
By Lemma \ref{lem:fiber of pol BR}, 
a point in the same fiber of ${\BRg}\to{\Dg}$ as $(S, L, \varphi)$ 
can be written as $(S', L', \varphi')=(S, L, w\circ \varphi)$ 
for some $w\in W(\omega, l)$. 
The Kuranishi family of $(S', L', \varphi')$ is 
$(\mathcal{X}'\to \mathcal{U}', \mathcal{L}', \varphi')=
(\mathcal{X}\to \mathcal{U}, \mathcal{L}, w\circ\varphi)$. 
Outside $\mathcal{H}$ 
the two marked families are isomorphic: 
the gluing is given by the equivariant action of $w^{-1}$ on $\mathcal{X}|_{\mathcal{U}^{\circ}}$, 
which we denote by 
\begin{equation*}
f : \mathcal{X}|_{\mathcal{U}^{\circ}} \to 
\mathcal{X}|_{\mathcal{U}^{\circ}} = \mathcal{X}'|_{(\mathcal{U}')^{\circ}}. 
\end{equation*}
It suffices to show that the closure of the graph of $f$ in 
$\mathcal{X}\times_{\mathcal{U}}\mathcal{X'}$ 
is an analytic set, 
and that its fiber over $u_{0}$, as a correspondence between $S$ and $S'$, 
gives an isomorphism between the complements of the exceptional divisors of 
$S\to \bar{S}$ and $S'\to \bar{S}'$. 
This would imply that $f$ extends to an isomorphism 
$\bar{\mathcal{X}}\backslash {\rm Sing}(\bar{\mathcal{X}}) \to 
\bar{\mathcal{X}}'\backslash {\rm Sing}(\bar{\mathcal{X}}')$ 
where $\bar{\mathcal{X}}$, $\bar{\mathcal{X}}'$ are 
the relative contractions of $\mathcal{X}$, $\mathcal{X}'$ respectively,  
which then extends to an isomorphism 
$\bar{\mathcal{X}} \to \bar{\mathcal{X}}'$ 
by normality. 

In order to prove our claim, we use a variant of the argument of Burns-Rapoport \cite{BR}. 
For each $u\in \mathcal{U}^{\circ}$, 
$f_{u}\colon \mathcal{X}_{u}\to \mathcal{X}_{w^{-1}u}$ satisfies 
\begin{equation*}
(f_u)_{\ast} = \varphi_{w^{-1}u}^{-1} \circ w^{-1} \circ \varphi_{u} : 
H^{2}(\mathcal{X}_{u}, {\Z}) \to H^{2}(\mathcal{X}_{w^{-1}u}, {\Z}). 
\end{equation*}
By the same argument as \cite{BR} p.248, the volume of the graph of $f_{u}$ is bounded. 
This shows that the graph of $f$ has finite volume, 
hence its closure is an analytic subset of 
$\mathcal{X}\times_{\mathcal{U}}\mathcal{X'}$ 
by \cite{Bi} Theorem 3. 
The limit cycle 
\begin{equation*}
\Gamma = \lim_{u\to u_{0}} {\rm graph}(f_{u})  
\; \; \subset \; \; S\times S'  
\end{equation*}
is of pure dimension $2$ and satisfies 
\begin{equation}\label{eqn: H2 action of limit cycle}
\Gamma_{\ast} = \varphi_{u_{0}}^{-1} \circ w^{-1} \circ \varphi_{u_{0}} : 
H^{2}(S, {\Z}) \to H^{2}(S', {\Z}).  
\end{equation}
By the same argument as \cite{BR} p.248 -- p.250, 
$\Gamma$ can be written as 
\begin{equation}\label{eqn:limit correspondence}
\Gamma = \textrm{graph}(g) + \sum_{i,j}a_{i,j}C_{i}\times C_{j}', 
\end{equation}
where $g$ is an isomorphism 
$S\to S'$, 
$a_{i,j}\geq0$, 
and $C_{i}, C_{j}'$ are irreducible curves on 
$S$, $S'$ respectively. 
What has to be shown is that 
$(C_{i}, L)=0$ and $(C'_{j}, L')=0$ 
whenever $a_{i,j}>0$. 
We use a variant of the argument of Matsusaka-Mumford 
in the proof of Theorem 2 of \cite{MM}. 

Let $U$ be the open set of $|3L|$ consisting of members $D\in |3L|$ 
such that $D\not\supset C_{k}$ for every $k$. 
Since $\Gamma_{\ast}([L])=[L']$, 
the correspondence by $\Gamma$ gives a birational map 
$\Gamma_{\ast}:|3L|\dashrightarrow |3L'|$, 
whose domain of definition contains $U$ 
and is injective over $U$ (\cite{MM} p.671).   
For $D\in U$ we have   
\begin{equation*}
\Gamma_{\ast}(D) = g(D) + \sum_{i,j}a_{i,j}(C_{i}, 3L)C_{j}'. 
\end{equation*}  
We suppose, to the contrary, that $(C_i, L)> 0$ for some $i$. 
For a general point $p$ of $C_{i}$ we set 
$U_{p}=U\cap |3L-p|$. 
Then $U_{p}$ is non-empty and of codimension $1$,  
so $\Gamma_{\ast}(U_p)$ is of codimension $1$ in $|3L'|$. 
If $a_{i,j}>0$, $\Gamma_{\ast}(U_{p})$ is contained in $|3L'-C_{j}'|+C_{j}'$. 
Since $3L'$ is base point free, 
we have $\dim |3L'-C_{j}'|<\dim |3L'|$, 
so $\Gamma_{\ast}(U_{p})$ is an open subset of $|3L'-C_{j}'|+C_{j}'$. 
Noticing that $p$ is any general point of $C_{i}$, 
we see that $U_{p}\cap U_{q}$ is open dense in both $U_{p}$ and $U_{q}$ 
for two general points $p, q$ of $C_i$. 
This contradicts that $3L$ separates two general points of $C_{i}$. 
Hence $(C_i, L)=0$ for every $i$. 
The assertion $(C'_j, L')=0$ can be proved similarly. 
\end{proof}

By construction, 
the isomorphism $g\colon S\to S'$ in \eqref{eqn:limit correspondence} 
is the minimal resolution of the gluing isomorphism $\bar{S}\to \bar{S}'$. 
As will be shown later (Corollary \ref{cor: Weyl group act trivially}), 
this coincides with the equivariant action by $w$.

\subsection{The ${\OLg}$-action}\label{ssec: OLg-action}

Next we prove that 
${\OLg}$ acts on ${\Xg}$ and determine its fixed divisor. 
Since the relative morphism ${\XgBR}\to {\proj}(\pi_{\ast}\mathcal{L}^{\otimes 3})^{\vee}$ 
is ${\OLg}$-equivariant, 
${\OLg}$ acts on ${\Xgcont}\to{\BRg}$ equivariantly. 

\begin{lemma}
The action of ${\OLg}$ on ${\Xgcont}\to{\BRg}$ 
descends to an equivariant action on ${\Xg}\to{\Dg}$. 
\end{lemma}

\begin{proof}
Let $\gamma\in{\OLg}$ and $[\omega]\in \mathcal{H}$. 
Let $[(S_{1}, \varphi_{1})]$, 
$[(S_{2}, \varphi_{2})]=[(S_{1}, w\circ \varphi_{1})]$ 
be two points in the fiber of ${\BRg}\to{\Dg}$ over $[\omega]$ where $w\in W(\omega, l)$,  
and $g\colon S_{1} \to S_{2}$ be the resolution of the gluing isomorphism between 
$\bar{S}_{1}$ and $\bar{S}_{2}$. 
Let  
$[(S_{1}', \varphi_{1}')]=\gamma[(S_{1}, \varphi_{1})]$, 
$w'=\gamma w \gamma^{-1}$, and 
$[(S_{2}', \varphi_{2}')]=[(S_{1}', w'\circ\varphi_{1}')]=\gamma[(S_{2}, \varphi_{2})]$. 
What has to be checked is that 
\begin{equation*}
\gamma \circ g \circ \gamma^{-1} : 
S_{1}' \to S_{1} \to S_{2} \to S_{2}' 
\end{equation*}
coincides with the resolution $g'$ of the gluing isomorphism between 
$\bar{S}_{1}'$ and $\bar{S}_{2}'$.  

Let 
$(\mathcal{X}\to\mathcal{U}, \mathcal{L}, \varphi_{i})$ 
and 
$(\mathcal{X}'\to\mathcal{U}', \mathcal{L}', \varphi_{i}')$
be Kuranishi families of 
$(S_{i}, L_{i}, \varphi_{i})$ 
and 
$(S_{i}', L_{i}', \varphi_{i}')$   
respectively.   
Then 
$\gamma \colon S_{i}\to S_{i}'$ extends to 
$\gamma \colon \mathcal{X}\to \mathcal{X}'$. 
Recall that 
${\rm graph}(g)$ and ${\rm graph}(g')$ 
are the main components of the limit cycles 
$\lim_{u}{\rm graph}(f_{u})$ and $\lim_{u'}{\rm graph}(f'_{u'})$ 
respectively, where 
$f$, $f'$ are the equivariant actions by $w^{-1}$, $(w')^{-1}$ 
on $\mathcal{X}, \mathcal{X}'$ respectively. 
Since 
$\gamma({\rm graph}(f_{u})) = {\rm graph}(f'_{\gamma u})$, 
we have   
\begin{equation*}
\lim_{u'}{\rm graph}(f'_{u'}) = 
\lim_{u} {\rm graph}(f'_{\gamma u}) = 
\lim_{u} \gamma({\rm graph}(f_{u})) = 
\gamma( \lim_{u}{\rm graph}(f_{u})). 
\end{equation*}
The main component of $\gamma( \lim_{u}{\rm graph}(f_{u}))$ 
is $\gamma({\rm graph}(g))$, 
which is the graph of $\gamma \circ g \circ \gamma^{-1}$. 
\end{proof}

We study the action of the $(-2)$-reflections. 

\begin{proposition}\label{prop: reflection on fiber}
Let $\delta$ be a $(-2)$-vector of ${\Lg}$, 
$s_{\delta}\in{\OLg}$ be the reflection with respect to $\delta$, 
and $\mathcal{H}_{\delta}=\delta^{\perp}\cap {\Dg}$. 
Then $s_{\delta}$ acts on ${\Xg}|_{\mathcal{H}_{\delta}}$ trivially. 
\end{proposition}

\begin{proof}
It suffices to show that $s_{\delta}$ acts trivially on the fiber 
over a general point $[\omega]$ of $\mathcal{H}_{\delta}$; 
then $s_{\delta}$ does so for all $[\omega]\in \mathcal{H}_{\delta}$ by continuity. 
When $[\omega]$ is general, we have 
$\Delta(\omega, l)=\{ \pm\delta \}$ and 
$W(\omega, l)= \{ {\rm id}, s_{\delta} \}$, 
so the fiber of ${\BRg}\to{\Dg}$ over $[\omega]$ consists of two points. 
Let $(S, \varphi)$ and $(S', \varphi')=(S, s_{\delta}\circ \varphi)$ be the corresponding 
marked $K3$ surfaces. 
What has to be shown is that 
the resolution $g\colon S\to S'$ of the gluing isomorphism 
$\bar{S}\to \bar{S}'$ coincides with 
the equivariant action by $s_{\delta}$. 
After the original identification $S'=S$, 
the action of $s_{\delta}$ is the identity of $S$. 
So it suffices to show that 
$g\colon S\to S'=S$ is also the identity of $S$. 

We may assume that $\varphi^{-1}(\delta)$ is effective, 
so $\varphi^{-1}(\delta)=[E]$ for a $(-2)$-curve $E$ on $S$. 
By the argument after \eqref{eqn:limit correspondence}, 
the limit cycle $\Gamma$ in $S\times S'= S\times S$ can be written as 
\begin{equation}\label{eqn:limit correspondence II}
\Gamma = {\rm graph}(g) + a E\times E  \; \;  \subset \; \; S\times S 
\end{equation} 
for some $a\geq0$. 
On the other hand, by \eqref{eqn: H2 action of limit cycle}, 
the action of $\Gamma$ on $H^{2}(S, {\Z})$ is  
the reflection $\varphi^{-1}\circ s_{\delta} \circ \varphi$ 
with respect to $E$. 
Comparing this with \eqref{eqn:limit correspondence II}, 
we see that 
\begin{equation*}
g_{\ast} = {\rm id} + (1-a) ( \cdot, E)E
\end{equation*}
on $H^{2}(S, {\Z})$. 
If we send $[E]$ by this action, we obtain $g(E)\sim (2a-1)E$, so $a=1$. 
Hence $g_{\ast}={\rm id}$. 
By the strong Torelli theorem, $g={\rm id}$. 
\end{proof}

\begin{corollary}\label{cor: Weyl group act trivially}
Let $[\omega]\in{\Dg}$. 
The Weyl group $W(\omega, l)$ acts trivially 
on the fiber of ${\Xg}\to{\Dg}$ over $[\omega]$. 
\end{corollary}

\begin{proof}
If $\delta\in \Delta(\omega, l)$, then $s_{\delta}$ acts trivially on ${\Xg}|_{\mathcal{H}_{\delta}}$ by Proposition \ref{prop: reflection on fiber}, 
which contains the fiber over $[\omega]$. 
Therefore every element of $W(\omega, l)$ acts trivially on the fiber over $[\omega]$. 
\end{proof}

\begin{remark}
When $g>2$, the divisor $\mathcal{H}_{\delta}$ parametrizes a family of 
lattice-polarized $K3$ surfaces whose very general Picard lattice is  
${\Z}L\oplus {\Z}E \simeq \langle 2g-2 \rangle \oplus \langle -2 \rangle$ 
or its overlattice of index $2$. 
Proposition \ref{prop: reflection on fiber} for $g>2$ also follows from 
the fact that no $K3$ surface with such a Picard lattice has an involution which acts trivially on the Picard lattice (\cite{Ni2}). 
\end{remark}

We can now determine the fixed divisor 
of the ${\OLg}$-action on ${\Xg}$. 
We need to separate the case $g=2$, 
where general members are double planes branched along sextics. 

\begin{proposition}\label{prop: fixed divisor}
Suppose that $\gamma\ne {\rm id} \in {\OLg}$ fixes a divisor of ${\Xg}$. 
Then, 

(1) If $g>2$, we have $\gamma=s_{\delta}$ for a $(-2)$-vector $\delta\in{\Lg}$. 

(2) If $g=2$, 
we have that, either 
$\gamma=s_{\delta}$ for a $(-2)$-vector $\delta\in \Lambda_{2}$, or 
$\gamma=-{\rm id}$ which acts as the covering transformation of the double planes. 
\end{proposition}

\begin{proof}
(1) Let $g>2$ and $\gamma\ne {\rm id}, s_{\delta}$ for any $(-2)$-vector $\delta$. 
We assume to the contrary that $\gamma$ fixes an irreducible divisor $D$ of ${\Xg}$. 
Since $g>2$, ${\OLg}$ does not contain $-{\rm id}$, 
so $\gamma$ acts nontrivially on ${\Dg}$. 
Then $D={\Xg}|_{B}$ for an irreducible divisor $B$ of ${\Dg}$ fixed by $\gamma$. 
Since $B$ is not contained in $\mathcal{H}_{\delta}$ for any $(-2)$-vector $\delta$, 
a general fiber $S$ of $D\to B$ is nonsingular, equipped with a marking $\varphi$. 
Then the equivariant action of $\gamma$ on $S$ 
is the identity of $S$ 
and satisfies $\varphi\circ {\rm id} = \gamma \circ \varphi$. 
This is absurd. 

(2) Let $g=2$. 
In this case $\tilde{{\rm O}}(\Lambda_{2})$ contains $-{\rm id}$. 
When $\gamma\ne \pm{\rm id}, \pm s_{\delta}$ for any $(-2)$-vector $\delta$, 
$\gamma$ fixes no divisor by the same argument as (1). 
The covering transformation of the double planes acts by $-{\rm id}$ on $\Lambda_{2}$ (cf.~\cite{Ni2}), 
so $\gamma=-{\rm id}$ acts on $\mathcal{X}_{2}^{\circ}$ by this involution. 
It remains to check that $\gamma=-s_{\delta}$ acts on 
$\mathcal{X}_{2}|_{\mathcal{H}_{\delta}}$ nontrivially. 
As in the proof of Proposition \ref{prop: reflection on fiber}, 
let $[(S, \varphi)], [(S', \varphi')]$ be the two points of $\tilde{\Omega}_{2}$ 
lying over a general point of $\mathcal{H}_{\delta}$. 
Both $s_{\delta}$ and $-{\rm id}$ send 
$[(S, \varphi)]$ to $[(S', \varphi')]$ in its action on $\tilde{\Omega}_{2}$, 
so $-s_{\delta}$ preserves $(S, \varphi)$ 
in its action on $\tilde{\mathcal{X}}_{2}$. 
Since $-s_{\delta}\ne {\rm id}$, this is a nontrivial involution of $S$. 
\end{proof}

When $g>2$, ${\OLg}$ also contains 
$-$reflections by splitting $(2-2g)$-vectors 
which fixes a divisor of ${\Dg}$ (cf.~\cite{GHS1}). 
But they act nontrivially 
on the smooth fibers over the fixed divisor, 
so does not fix a divisor of ${\Xg}$.

Before moving on, 
we also study the relation with the automorphism group. 
For $[\omega]\in{\Dg}$ 
let $G(\omega, l)$ be the stabilizer of $[\omega]$ in ${\OLg}$. 

\begin{proposition}\label{prop: stabilizer and auto}
Let $(S, L)$ be the polarized fiber of ${\Xg}\to{\Dg}$ over $[\omega]$. 
We have a split exact sequence 
\begin{equation}\label{eqn: stabilizer and auto group}
0 \to W(\omega, l) \to G(\omega, l) \stackrel{\pi}{\to} {\rm Aut}(S, L) \to 0 
\end{equation}
where $\pi$ is induced from the equivariant action of $G(\omega, l)$ on ${\Xg}$. 
A splitting of \eqref{eqn: stabilizer and auto group} is obtained 
for each choice of a point of ${\BRg}$ over $[\omega]$. 
\end{proposition}
 
\begin{proof}
By Corollary \ref{cor: Weyl group act trivially}, $W(\omega, l)$ acts trivially on $S$. 
We take a point $([\omega], \mathcal{C})$ of ${\BRg}$ lying over $[\omega]$, 
and let $(\tilde{S}, \tilde{L}, \varphi)$ be the corresponding marked quasi-polarized $K3$ surface. 
Then $\tilde{S}$ is the minimal resolution of $S$ 
and $\tilde{L}$ is the pullback of $L$.  
Then ${\rm Aut}(S, L)$ is isomorphic to ${\rm Aut}(\tilde{S}, \tilde{L})$, 
which in turn is isomorphic to the stabilizer of $([\omega], \mathcal{C})$ in ${\OLg}$ 
by the strong Torelli theorem. 
The last group is the stabilizer of the chamber $\mathcal{C}$ in $G(\omega, l)$. 
This defines a section ${\rm Aut}(S, L)\hookrightarrow G(\omega, l)$ of $\pi$. 
In view of Lemma \ref{lem:fiber of pol BR},  
this also shows that \eqref{eqn: stabilizer and auto group} is exact at the middle. 
\end{proof}

\subsection{${\Fgn}$}\label{ssec: Fgn}

Let $n\geq 1$. 
We take the $n$-fold fiber product 
\begin{equation*}
{\Xgn} = {\Xg} \times_{{\Dg}} \cdots \times_{{\Dg}} {\Xg} 
\end{equation*}
and its quotient 
\begin{equation*}
{\Fgn} = {\OLg}\backslash {\Xgn}. 
\end{equation*}
Since ${\OLg}$ acts properly discontinuously on ${\Dg}$, 
it does so on ${\Xgn}$. 
Thus ${\Fgn}$ is an irreducible complex analytic space fibered over ${\Fg}$.  
Since ${\Xgn}$ is a flat family of normal varieties over the smooth base ${\Dg}$, 
${\Xgn}$ is normal by \cite{EGA4} Corollaire 6.5.4. 
Therefore ${\Fgn}$ is also normal. 

If $(S, L)$ is the (possibly singular) polarized $K3$ surface over $[\omega]\in{\Dg}$, 
the fiber of ${\Fgn}\to {\Fg}$ over $[\omega]\in{\Fg}$ is $S^{n}/{\rm Aut}(S, L)$ 
by Proposition \ref{prop: stabilizer and auto}. 
Thus ${\Fgn}$ is the moduli space 
of $n$-pointed polarized $K3$ surfaces of genus $g$ 
with at worst rational double points. 

\begin{proposition}\label{prop: ramification divisor}
Let $\pi\colon {\Xgn}\to{\Dg}$ be the projection. 

(1) Let $(g, n)\ne (2, 1)$. 
Then ${\Xgn}\to{\Fgn}$ is doubly ramified at $\pi^{-1}(\mathcal{H}_{\delta})$ 
for $(-2)$-vectors $\delta\in {\Lg}$ and has no other ramification divisor. 
The branch divisor of ${\Xgn}\to{\Fgn}$ is irreducible when $g\not\equiv 2$ mod $4$, 
and consists of two irreducible components when $g\equiv 2$ mod $4$.   

(2) Let $(g, n)=(2, 1)$. 
The ramification divisor of $\mathcal{X}_{2,1}\to \mathcal{F}_{2,1}$ consists of 
$\pi^{-1}(\mathcal{H}_{\delta})$ for $(-2)$-vectors $\delta\in {\Lg}$ 
and the closure of the locus of 
ramification curves of the double planes. 
The branch divisor consists of three irreducible components.  
\end{proposition}

\begin{proof}
The ramification divisors are determined by Proposition \ref{prop: fixed divisor}. 
Classification of the irreducible components of the branch divisor reduces to 
that of ${\OLg}$-equivalence classes of $(-2)$-vectors in ${\Lg}$, 
which is given in \cite{GHS5} Proposition 2.4 (ii). 
\end{proof}

As our construction is done over the period domain  
(which is necessary for the connection with modular forms), 
${\Fgn}$ is a priori just complex analytic. 
But actually 

\begin{proposition}\label{prop: quasi-proj}
${\Fgn}$ is quasi-projective. 
\end{proposition}

\begin{proof}
We use the Hilbert scheme construction. 
Since the basic line of argument is similar to 
\cite{Vi} Chapter 7 and \cite{Hu} Chapter 5 
(where ${\OLg}\backslash \Omega_{g}^{\circ}$ is considered), 
we will be brief. 

Let $H$ be the Hilbert scheme of polarized $K3$ surfaces $(S, L)$ of degree $2g-2$ 
with at worst rational double points, projectively embedded by $L^{\otimes 3}$ 
into ${\proj}^{N}$ where $N=18g-17$. 
The first step is to show that $H$ is nonsingular. 
In the locus of smooth $K3$ surfaces, this is proved in \cite{Hu} Chapter 5.3.1. 
In the ADE locus, 
what has to be shown is that 
the obstruction space ${\rm Ext}^{1}(I_{S}/I_{S}^{2}, \mathcal{O}_{S})$ is trivial, 
where $I_{S}$ is the ideal sheaf of $S\subset {\proj}^{N}$. 
Although $S$ is singular, 
its singularities are complete intersections,  
so $I_{S}/I_{S}^{2}$ is locally free and sits in the short exact sequence 
\begin{equation*}
0 \to I_{S}/I_{S}^{2} \to \Omega_{{\proj}^{N}}^{1}|_{S} \to \Omega_{S}^{1} \to 0 
\end{equation*}
by \cite{Ei} Exercise 16.17. 
Taking the ${\rm Ext}(\cdot, \mathcal{O}_{S})$ long exact sequence, 
we obtain the exact sequence 
\begin{equation*}
{\rm Ext}^{1}(\Omega_{S}^{1}, \mathcal{O}_{S}) \stackrel{\phi}{\to} 
{\rm Ext}^{1}(\Omega_{{\proj}^{N}}^{1}|_{S}, \mathcal{O}_{S}) \to 
{\rm Ext}^{1}(I_{S}/I^{2}_{S}, \mathcal{O}_{S})  
 \to  {\rm Ext}^{2}(\Omega_{S}^{1}, \mathcal{O}_{S}). 
\end{equation*}
By the same calculation as in \cite{Hu} p.92, 
we see that 
\begin{equation*}
{\rm Ext}^{2}(\Omega_{S}^{1}, \mathcal{O}_{S}) 
\simeq H^{0}(S, \Omega_{S}^{1})^{\vee} = 0 
\end{equation*} 
and that the Serre dual of $\phi$, 
\begin{equation*}
H^{1}(S, \Omega_{{\proj}^{N}}^{1}|_{S}) \simeq {\C} 
 \to H^{1}(S, \Omega_{S}^{1}) 
\end{equation*}
is injective. 
Therefore 
${\rm Ext}^{1}(I_{S}/I_{S}^{2}, \mathcal{O}_{S})=0$. 

The next step is to apply the argument of \cite{Vi} \S 7.3 II. 
By construction we have a flat projective family $\frak{X}\to H$ 
of $K3$ surfaces with at worst rational double points, 
acted on by $G={\rm PGL}_{N+1}$. 
According to loc.~cit., with the smoothness of $H$, 
there exists 
a finite group $\Gamma$, 
quasi-projective varieties $V, Z$ with $\Gamma\times G$ acting on $V$, 
and morphisms $V\to H$, $V\to Z$, 
such that 
$V\to H$ is the quotient by $\Gamma$ 
and $V\to Z$ is a principal $G$-bundle. 
Furthermore, there exists a geometric quotient $X'=(\frak{X}\times_{H}V)/G$ 
with a projective morphism $X'\to Z$. 
In particular, the $n$-fold fiber product 
$X_{n}'=X'\times_{Z} \cdots \times_{Z}X'$ 
is quasi-projective. 
By \cite{Vi} Corollary 3.46.2, 
a geometric quotient $X'_{n}/\Gamma$ exists as a quasi-projective variety. 
The variety $X'_{n}/\Gamma$ is also a geometric quotient of 
$\frak{X}\times_{H}\cdots \times_{H}\frak{X}$ by $G$, 
which in turn is naturally biholomorphic to ${\Fgn}$. 
Therefore ${\Fgn}$ is quasi-projective. 
\end{proof}


\section{Modular forms}\label{sec: modular form}

In this section, we let $\Lambda$ be a general integral lattice of signature $(2, b)$ with $b\geq3$. 
(We will specialize to $\Lambda={\Lg}$ in the next section.) 
We recall the basic theory of modular forms for the orthogonal group of $\Lambda$. 
The material should be well-known, 
but since we could not find a suitable reference at some points, 
we tried to be rather self-contained. 
The results needed in \S \ref{sec: theorem 1} are 
Corollary \ref{cor: vanish at H} and Proposition \ref{prop: cusp form Petersson}.

\subsection{Modular forms}\label{ssec: modular forms}

We choose one of the two connected components of 
\begin{equation*}
\{ \: [\omega]\in{\proj}\Lambda_{{\C}} \: | \: 
(\omega, \omega)=0, (\omega, \bar{\omega})>0 \: \} 
\end{equation*} 
and denote it by ${\D}$. 
Let ${\OL}$ be the index $\leq 2$ subgroup of ${\rm O}(\Lambda)$ preserving ${\D}$. 
We write $\lambda$ for the restriction of 
$\mathcal{O}_{{\proj}\Lambda_{{\C}}}(-1)$ to ${\D}$. 
Then ${\OL}$ acts on $\lambda$ equivariantly. 
Let ${\G}$ be a finite-index subgroup of ${\OL}$ 
and $\chi$ be a unitary character of ${\G}$. 
The character $\chi$ corresponds to a ${\G}$-linearized line bundle on ${\D}$ 
which we again denote by $\chi$. 
In later sections, $\chi$ will be either trivial or the determinant character 
${\det}\colon {\G}\to \{ \pm 1 \}$. 

\begin{example}
We have the canonical isomorphism of line bundles  
\begin{equation}\label{eqn: KD}
K_{{\D}} \simeq \lambda^{\otimes b}\otimes {\det} 
\end{equation}
(cf.~\cite{GHS1}). 
This is a consequence of the Euler sequence for ${\proj}\Lambda_{{\C}}$ 
and the adjunction formula for the isotropic quadric. 
\end{example}

A ${\G}$-invariant holomorphic section of the line bundle 
$\lambda^{\otimes k}\otimes \chi$ 
on ${\D}$ is called a \textit{modular form} of weight $k$ and character $\chi$ 
with respect to ${\G}$. 
We denote by $M_{k}({\G}, \chi)$ the space of such modular forms. 
As in \eqref{eqn: (-2)-Heegner K3}, 
the $(-2)$-Heegner divisor, 
now restricted to the component $\mathcal{D}$, 
is defined by 
\begin{equation*}
\mathcal{H} = 
\bigcup_{\begin{subarray}{c} \delta\in \Lambda \\ (\delta, \delta)=-2 \end{subarray}} 
\delta^{\perp}\cap{\D}. 
\end{equation*}
For a natural number $m$ we denote by  
$M_{k}({\G}, \chi)^{(m)}\subset M_{k}({\G}, \chi)$ 
the subspace of modular forms 
which vanish to order $\geq m$ along every irreducible component of $\mathcal{H}$. 
When ${\G}$ is torsion-free, 
this is identified with 
\begin{equation*}
M_{k}({\G}, \chi)^{(m)} = 
H^{0}({\G}\backslash {\D}, \lambda^{\otimes k}\otimes \chi (-mH) ), 
\end{equation*}
where 
$\lambda, \chi$ in the right hand side stand for the descends of 
the linearized line bundles $\lambda, \chi$ to ${\G}\backslash{\D}$, 
and $H$ is the image of $\mathcal{H}$ in ${\G}\backslash{\D}$.  

\begin{lemma}\label{lem: reflection LB}
Let $F\in M_{k}(\Gamma, \det)$. 
Let $\delta\in \Lambda$ be a vector of negative norm 
and $s_{\delta}\in {\rm O}^{+}(\Lambda_{{\Q}})$ 
be the reflection by $\delta$. 
Suppose that either 
(i) $s_{\delta}\in \Gamma$ or 
(ii) $-s_{\delta}\in \Gamma$ and $k \equiv b$ mod $2$. 
Then $F$ vanishes at $\delta^{\perp}\cap \mathcal{D}$ 
with odd vanishing order. 
\end{lemma}

\begin{proof}
If $[\omega]\in \delta^{\perp}\cap \mathcal{D}$, 
then $s_{\delta}$ acts trivially on $\lambda_{[\omega]}={\C}\omega$ by definition. 
Since $s_{\delta}$ has determinant $-1$, 
it acts on $(\lambda^{\otimes k}\otimes \det)_{[\omega]}$ by $-1$. 
Thus, if $F$ is $s_{\delta}$-invariant, $F([\omega])$ must be zero. 
When $k\equiv b$ mod $2$, $-{\rm id}$ acts on 
$\lambda^{\otimes k}\otimes \det$ by $(-1)^{k+b}=1$, 
so $-s_{\delta}$ acts on 
$(\lambda^{\otimes k}\otimes \det)_{[\omega]}$ by $-1$. 
\end{proof}

This is a straightforward generalization of the calculation in 
\cite{GHS1} Lemma 5.2 and \cite{GHS3} p.420.  
(It seems that in \cite{GHS3} p.420 
one has to assume $a\in 2{\Z}$ when $-\sigma$ is a reflection.) 

Let $\tilde{{\rm O}}^{+}(\Lambda)$ be the kernel of 
${\rm O}^{+}(\Lambda)\to {\rm O}(\Lambda^{\vee}/\Lambda)$. 
Since $\tilde{{\rm O}}^{+}(\Lambda)$ contains $s_{\delta}$ 
for all $(-2)$-vectors $\delta\in \Lambda$, 
we have 

\begin{corollary}\label{cor: vanish at H}
$M_{k}(\tilde{{\rm O}}^{+}(\Lambda), \det)^{(1)} 
= M_{k}(\tilde{{\rm O}}^{+}(\Lambda), \det)$. 
\end{corollary}

Thus, for $\Gamma=\tilde{{\rm O}}^{+}(\Lambda)$ and $\chi=\det$, 
vanishing at $\mathcal{H}$ is automatic. 
For general $\Gamma$ and $\chi=\det$, 
vanishing at the ramification divisor of 
$\mathcal{D}\to \Gamma\backslash \mathcal{D}$ 
is automatic when $k\equiv b$ mod $2$. 

Next we explain the Petersson metric and the invariant volume form. 
The Hermitian form $( \cdot, \bar{\cdot}  )$ on $\Lambda_{{\C}}$ 
defines a Hermitian metric on the line bundle $\lambda$, 
whose positivity follows from the definition of ${\D}$. 
By construction this is ${\OLR}$-invariant. 
On the other hand, since the character $\chi$ is unitary, 
we also have a ${\G}$-invariant Hermitian metric on the corresponding line bundle. 
These define a ${\G}$-invariant Hermitian metric on the line bundle 
$\lambda^{\otimes k}\otimes \chi$  
which we call the (pointwise) Petersson metric on $\lambda^{\otimes k}\otimes \chi$. 

By the isomorphism \eqref{eqn: KD}, 
the Petersson metric on $\lambda^{\otimes b}\otimes {\det}$ 
defines an ${\rm O}^{+}(\Lambda_{{\R}})$-invariant Hermitian metric on $K_{{\D}}$ 
which we denote by $( \: , \: )_{K}$. 
We define a volume form ${\volD}$ on ${\D}$ by 
\begin{equation}\label{eqn: volD}
({\volD})_{[\omega]} = 
i^{b^{2}} \frac{\Omega \wedge \bar{\Omega}}{(\Omega, \Omega)_{K}}, 
\qquad [\omega]\in{\D}, 
\end{equation}
where $\Omega$ is an arbitrary nonzero vector of $(K_{{\D}})_{[\omega]}$. 
This does not depend on the choice of $\Omega$. 
Since $( \: , \: )_{K}$ is ${\rm O}^{+}(\Lambda_{{\R}})$-invariant, so is ${\volD}$.

\subsection{Tube domain realization}\label{ssec: tube domain}

We recall, in some detail, 
tube domain realization of ${\D}$ and Fourier expansion of modular forms. 
This is necessary for the proof of Propositions \ref{prop: cusp form Petersson} and 
\ref{prop: extension conditional}. 

Let $I$ be a rank $1$ primitive isotropic sublattice of $\Lambda$. 
This corresponds to a $0$-dimensional rational boundary component of ${\D}$ 
(cf.~\cite{BB}, \cite{Sc}). 
We put ${\LI}=I\otimes(I^{\perp}/I)$. 
The quadratic form on $I^{\perp}/I$ and an isomorphism $I\simeq {\Z}$ 
(unique up to $\pm1$) define a canonical quadratic form on ${\LI}$. 
We write ${\GIQ}$ for the stabilizer of $I_{{\Q}}$ in 
${\rm O}^{+}(\Lambda_{{\Q}})$. 
The unipotent radical ${\UIQ}$ of ${\GIQ}$ 
consists of the Eichler transvections (see \cite{Sc}, \cite{GHS2}) 
\begin{equation*}
E_{l,m} : v \mapsto v - (m, v)l +(l, v)m - \frac{1}{2}(m, m)(l, v)l, 
\quad v\in\Lambda_{{\Q}}, 
\end{equation*}
where 
$l\in I_{{\Q}}$ and $m\in I^{\perp}_{{\Q}}$. 
This induces a canonical isomorphism 
\begin{equation*}
{\LIQ} \to {\UIQ}, \quad l\otimes \bar{m} \mapsto E_{l,m}, 
\end{equation*}
where $m\in I_{{\Q}}^{\perp}$ is a lift of $\bar{m}\in (I^{\perp}/I)_{{\Q}}$.

The projection 
${\proj}\Lambda_{{\C}}\dashrightarrow {\proj}(\Lambda/I)_{{\C}}$ 
from 
${\proj}I_{{\C}}\in {\proj}\Lambda_{{\C}}$ 
defines an embedding 
\begin{equation}\label{eqn: tube domain}
{\D} \hookrightarrow {\proj}(\Lambda/I)_{{\C}} - {\proj}(I^{\perp}/I)_{{\C}}.  
\end{equation} 
If we choose a line $I'_{{\Q}}\subset \Lambda_{{\Q}}$ with $(I_{{\Q}}, I'_{{\Q}})\ne 0$, 
this defines an isomorphism 
\begin{equation*}
{\proj}(\Lambda/I)_{{\C}} - {\proj}(I^{\perp}/I)_{{\C}} \simeq 
{\hom}(I'_{{\C}}, (I^{\perp}/I)_{{\C}}) \simeq {\LIC}. 
\end{equation*}
Thus \eqref{eqn: tube domain} induces an embedding 
\begin{equation*}
\iota_{I'} : {\D} \hookrightarrow {\LIC}. 
\end{equation*}
The choice of the component ${\D}$  
determines a connected component of 
$\{ v\in {\LIR} | (v, v)>0 \}$ 
which we denote by $\mathcal{C}_{I}$. 
If we put 
\begin{equation*}
\mathcal{D}_{I} = \{ \: Z\in {\LIC} \: | \: {\rm Im}(Z) \in \mathcal{C}_{I} \: \},  
\end{equation*}
then $\iota_{I'}({\D})=\mathcal{D}_{I}$. 
This is a tube domain realization of ${\D}$. 
Via $\iota_{I'}$, 
${\UIQ}\simeq{\LIQ}$ acts on $\mathcal{D}_{I}\subset {\LIC}$ 
by translation. 

Let ${\G}$ be a finite-index subgroup of ${\rm O}^{+}(\Lambda)$. 
We set ${\GIZ}={\GIQ}\cap {\G}$ and ${\UIZ}={\UIQ}\cap {\G}$. 
Then ${\UIZ}$ is a lattice on ${\UIQ}$. 
For example, 
when $\Lambda$ is even and ${\G}=\tilde{{\rm O}}^{+}(\Lambda)$, 
we have ${\UIZ}={\LI}$. 
Let $\chi$ be a unitary character of ${\G}$. 
We assume that $\chi|_{{\UIZ}}=1$. 
(This is always satisfied for $\chi={\det}$.) 
We choose a generator $l$ of $I$. 
The homogeneous function $(\cdot , l)^{-1}$ on $\Lambda_{{\C}}$ of degree $-1$ 
defines a frame $s_{l}$ of $\lambda$. 
Its factor of automorphy is 
\begin{equation*}
j(\gamma, [\omega]) 
= \frac{(\gamma\omega, l)}{(\omega, l)} 
= \frac{(\omega, \gamma^{-1}l)}{(\omega, l)}, \qquad 
\gamma\in\Gamma, \; \;  [\omega]\in \mathcal{D}. 
\end{equation*}
We also choose a nonzero vector $w_{0}$ in the representation line of $\chi$. 
Then modular forms $F=f(s_{l}^{\otimes k}\otimes w_{0})$ of weight $k$ and character $\chi$ for ${\G}$ 
are identified with holomorphic functions $f$ on ${\D}$ satisfying 
\begin{equation*}
f(\gamma[\omega]) = \chi(\gamma) j(\gamma, [\omega])^{k}f([\omega]), 
\qquad [\omega]\in {\D}, \: \gamma \in {\G}. 
\end{equation*}
By our assumption $s_{l}^{\otimes k}\otimes w_{0}$ is invariant under ${\UIZ}$, 
so the function $f$ is ${\UIZ}$-invariant. 
If we regard $f$ as a function on $\mathcal{D}_{I}$ via $\iota_{I'}$, 
it is invariant under translation by the lattice ${\UIZ}$. 
Hence it admits a Fourier expansion 
\begin{equation}\label{eqn: Fourier expansion}
f(Z) = 
\sum_{m\in U(I)_{{\Z}}^{\vee}} a_{m}{\exp}(2\pi i(m, Z)), 
\qquad Z\in \mathcal{D}_{I}, 
\end{equation}
where $U(I)_{{\Z}}^{\vee}\subset {\UIQ}$ is the dual lattice of ${\UIZ}$. 
Let $\mathcal{C}_{I}^{+}$ be the union of $\mathcal{C}_{I}$ 
and the rays ${\R}_{\geq0}v$ for rational isotropic vectors $v\in {\UIQ}$ 
in the closure of $\mathcal{C}_{I}$. 
The Koecher principle says that $a_{m}\ne 0$ only when $m\in \mathcal{C}_{I}^{+}$. 
The modular form $F$ is called a \textit{cusp form} if $a_{m}=0$ for all $m$ with $(m, m)=0$ 
at all primitive rank $1$ isotropic sublattices $I$. 

Over $\mathcal{D}_{I}$ 
the Petersson metric and the invariant volume form 
are expressed as follows 
(compare with \cite{Br} p.96 for (2)).  

\begin{lemma}\label{prop: Petersson tube domain}
We choose a generator $l$ of $I$ and 
a line $I'_{{\Q}}\subset \Lambda_{{\Q}}$ with $(I_{{\Q}}, I'_{{\Q}})\ne 0$. 
We identify ${\D}\simeq \mathcal{D}_{I}$ by the tube domain realization $\iota_{I'}$. 

(1) Let $( \: , \: )_{1}$ be the Petersson metric on $\lambda$. 
Then we have up to a constant 
\begin{equation*}
(s_{l}(Z), s_{l}(Z))_{1} = ({\rm Im}(Z), {\rm Im}(Z)), 
\quad Z\in \mathcal{D}_{I}. 
\end{equation*}

(2) Let ${\rm vol}_{I}$ be a flat volume form on ${\LIC}$. 
Then we have up to a constant 
\begin{equation*} 
{\volD} = \frac{1}{({\rm Im}(Z), {\rm Im}(Z))^{b}} {\rm vol}_{I}. 
\end{equation*}

(3) Let 
$F_{1}=f_{1}(s_{l}^{\otimes k}\otimes w_{0})$ and 
$F_{2}=f_{2}(s_{l}^{\otimes k}\otimes w_{0})$ be local sections of 
$\lambda^{\otimes k}\otimes \chi$ where 
$f_{1}, f_{2}$ are local functions on ${\D}\simeq \mathcal{D}_{I}$. 
Let $(\: , \: )_{k,\chi}$ be the Petersson metric on $\lambda^{\otimes k}\otimes \chi$. 
Then we have up to a constant 
\begin{equation*}
(F_{1}(Z), F_{2}(Z))_{k,\chi} {\volD} = 
f_{1}(Z)\overline{f_{2}(Z)} ({\rm Im}(Z), {\rm Im}(Z))^{k-b} {\rm vol}_{I}, 
\quad Z\in \mathcal{D}_{I}. 
\end{equation*} 
\end{lemma}

\begin{proof}
(1) 
The choice of $l$ defines an isomorphism ${\LIQ}\simeq (I^{\perp}/I)_{{\Q}}$, 
and then that of $I_{{\Q}}'$ defines a lift  
$(I^{\perp}/I)_{{\Q}}\simeq (I_{{\Q}}\oplus I_{{\Q}}')^{\perp} \subset \Lambda_{{\Q}}$ 
of ${\LIQ}$ in $\Lambda_{{\Q}}$. 
If we take the isotropic vector $l'$ in $I_{{\Q}}\oplus I_{{\Q}}'$ with $(l, l')=1$, 
we have 
\begin{equation*}
s_{l}(Z)=l' + Z - \frac{(Z, Z)}{2}l \; \in \; \Lambda_{{\C}} 
\end{equation*}
for $Z\in \mathcal{D}_{I}\subset {\LIC}$. 
Hence 
\begin{equation*}
(s_{l}(Z), s_{l}(Z))_{1} 
= (s_{l}(Z), \overline{s_{l}(Z)}) 
= -{\rm Re}(Z, Z) + (Z, \overline{Z}) 
= 2 ({\rm Im}(Z), {\rm Im}(Z)). 
\end{equation*}

(2) 
We take coordinates $z_{1}, \cdots , z_{b}$ on ${\LIC}$. 
By the isomorphism \eqref{eqn: KD}, the flat canonical form 
$\Omega=dz_{1}\wedge \cdots \wedge dz_{b}$ 
corresponds to $s_{l}^{\otimes b}\otimes w_{0}$ up to a constant 
because both are ${\UIC}$-invariant. 
Then we have 
\begin{equation*}
{\volD} 
= i^{b^{2}} \frac{\Omega \wedge \bar{\Omega}}{(\Omega, \Omega)_{K}} 
= \frac{{\rm vol}_{I}}{(s_{l}(Z), s_{l}(Z))_{1}^{b}} 
= \frac{1}{({\rm Im}(Z), {\rm Im}(Z))^{b}} {\rm vol}_{I} 
\end{equation*}
by (1). 
Assertion (3) follows from (1) and (2). 
\end{proof}

\subsection{Cusp forms and Petersson norm}\label{ssec: cusp Petersson}

Cusp forms of weight $\geq b$ can be characterized among modular forms 
by convergence of the global Petersson norm. 
This is used in the proof of the isomorphism \eqref{main eqn 2}. 

\begin{proposition}\label{prop: cusp form Petersson}
Let $F$ be a modular form of weight $k\geq b$ and character $\chi$ for ${\G}$. 
Then $F$ is a cusp form if and only if 
$\int_{{\G}\backslash{\D}}(F, F)_{k,\chi}{\volD} < \infty$. 
\end{proposition}

\begin{proof}
This should be a standard fact: 
see, e.g., \cite{Fr} Chapter III, \S 2 for 
a similar result in the case of ${\rm Sp}(2g, {\Z})$ and weight $g+1$. 
But since we could not find a suitable reference for the present case, 
we give an outline of the proof for the sake of completeness. 

It is convenient to work with a toroidal compactification (\cite{AMRT}). 
The input for constructing a toroidal compactification of ${\G}\backslash{\D}$ 
is a collection $\Sigma=(\Sigma_{I})_{I}$ of 
${\GIZ}$-admissible cone decomposition $\Sigma_{I}$ of $\mathcal{C}_{I}^{+}$, 
one for each ${\G}$-equivalence class of rank $1$ primitive isotropic sublattices $I$ 
(independently). 
We can take $\Sigma_{I}$ to be regular with respect to ${\UIZ}$. 
Let $\mathcal{T}_{I}={\UIZ}\backslash{\D}$. 
Via $\iota_{I'}$, 
$\mathcal{T}_{I}$ is realized as an open set of the torus 
${\UIZ}\backslash{\UIC}$, 
and the fan $\Sigma_{I}$ defines a toroidal embedding 
$\mathcal{T}_{I}\hookrightarrow \mathcal{T}_{I}^{\Sigma}$. 
This is the partial compactification in the direction of the $0$-dimensional cusp $I$. 
The domain ${\D}$ has also $1$-dimensional cusps, 
but the partial compactification there is canonical 
and has etale glueing maps to the partial compactifications 
at the adjacent $0$-dimensional cusps $I$. 
(This corresponds to a boundary ray in $\mathcal{C}_{I}^{+}$.) 
The toroidal compactification 
$({\G}\backslash{\D})^{\Sigma}$ of ${\G}\backslash{\D}$ 
is obtained by glueing ${\G}\backslash{\D}$ and 
quotients of the partial compactifications, 
essentially along small neighborhoods of the boundary. 
For a cone $\sigma\in\Sigma_{I}$, 
we write 
$\mathcal{T}_{I}^{\sigma}\subset \mathcal{T}_{I}^{\Sigma}$ 
for the toroidal embedding by $\sigma$, and let 
$U=U_{x}\subset \mathcal{T}_{I}^{\sigma}$ 
be a small neighborhood of an arbitrary point $x$ of 
the (snc) boundary divisor of $\mathcal{T}_{I}^{\sigma}$. 
Since $({\G}\backslash{\D})^{\Sigma}$ is compact,  
we have 
$\int_{{\G}\backslash{\D}}(F, F)_{k,\chi}{\volD} < \infty$ 
if and only if 
$\int_{U}(F, F)_{k,\chi}{\volD} < \infty$ 
for all $I$, $\sigma$, $x$.

We show that $F$ is a cusp form if and only if 
$\int_{U}(F, F)_{k,\chi}{\volD} < \infty$ 
for all $U$ with $\dim \sigma =1$. 
Let $\sigma={\R}_{\geq0}v_{0}$ with $v_{0}\in{\UIZ}$ primitive. 
We take a vector $m_{0}\in U(I)_{{\Z}}^{\vee}$ with $(v_{0}, m_{0})=1$ and put 
$q={\exp}(2\pi i(Z, m_{0}))$. 
Then $q$ gives a normal parameter around the boundary divisor of $\mathcal{T}_{I}^{\sigma}$. 
We write $F(Z)=f(Z)(s_{l}^{\otimes k}\otimes w_{0})$. 
We set $a=2$ if $(v_{0}, v_{0})>0$ and $a=1$ if $(v_{0}, v_{0})=0$. 
By Lemma \ref{prop: Petersson tube domain} (3) 
we have up to a constant 
\begin{eqnarray*}\label{eqn: asymptotic estimate}
(F(Z), F(Z))_{k,\chi}{\volD} 
& = & 
|f(Z)|^{2}({\rm Im}(Z), {\rm Im}(Z))^{k-b} {\rm vol}_{I} \\ 
& \sim & 
|f(Z)|^{2} (\log |q|)^{a(k-b)} |q|^{-2} dq \wedge d\bar{q} \wedge {\rm vol}'_{I} \\ 
& = & 
|f(Z)|^{2} (\log r)^{a(k-b)} r^{-1} dr \wedge d\theta \wedge {\rm vol}'_{I}
\end{eqnarray*}
as $|q|\to 0$, 
where $q=re^{i\theta}$ and 
${\rm vol}'_{I}$ is a flat volume form on the boundary divisor. 
Since $k\geq b$, this shows that 
$\int_{U}(F, F)_{k,\chi}{\volD} < \infty$ 
if and only if 
$f(Z)\to 0$ as $|q|\to 0$. 
In the Fourier expansion \eqref{eqn: Fourier expansion}, 
this is equivalent to $a_{m}=0$ for all $m$ with $(m, v_{0})=0$. 
When $(v_{0}, v_{0})>0$, this is equivalent to $a_{0}=0$ 
because $v_{0}^{\perp}\cap \mathcal{C}_{I}^{+} = \{ 0 \}$. 
When $(v_{0}, v_{0})=0$, we have $v_{0}^{\perp} \cap \mathcal{C}_{I}^{+} = \sigma$, 
so this is equivalent to $a_{m}=0$ 
for all $m\in \sigma\cap U(I)_{{\Z}}^{\vee}$. 
Since every boundary ray of $\mathcal{C}_{I}^{+}$ is also a ray in $\Sigma_{I}$, 
this implies the above equivalence and so the ``if '' part of the proposition. 

Conversely, if $F$ is a cusp form, 
we can check $\int_{U}(F, F)_{k,\chi}{\volD} < \infty$ 
for $U$ with $\dim \sigma$ general 
by a similar asymptotic estimate. 
\end{proof}

\section{Pluricanonical forms and modular forms}\label{sec: theorem 1}

In this section we prove Theorem \ref{main thm 1} 
based on the results of \S \ref{sec:2} and \S \ref{sec: modular form}. 
The proof will be completed at \S \ref{ssec: proof thm 1}. 
In \S \ref{ssec: main thm 2} we discuss 
extension over a certain compactification of ${\Fgn}$, 
assuming some properties. 

For the consistency with \S \ref{sec: modular form}, 
we need to restrict attention to a connected component. 
We choose a component $\mathcal{D}_{g}$ of ${\Dg}$ 
and put ${\Gg}={\OLg}\cap {\rm O}^{+}({\Lg})$. 
Let $\mathcal{X}_{g,n}^{+}$ be the connected component of ${\Xgn}$ over $\mathcal{D}_{g}$ 
and $\pi\colon \mathcal{X}_{g,n}^{+}\to \mathcal{D}_{g}$ be the projection. 
Since $\tilde{{\rm O}}({\Lg})$ contains an element exchanging the two components of ${\Dg}$, 
we have ${\Fgn} = {\Gg}\backslash \mathcal{X}_{g,n}^{+}$. 
We write $\lambda$ for the restriction of the tautological bundle $\mathcal{O}(-1)$ 
to ${\mathcal{D}_{g}}$ (or to ${\Dg}$).

\subsection{Relative canonical forms and modular forms}\label{ssec: Krelative}

First we clarify the relation between 
relative canonical forms on $\mathcal{X}_{g,n}^{+}$ and modular forms. 
The main attention will be on the degenerate family over $\mathcal{H}$. 
We begin with the following remark. 

\begin{lemma}
The canonical divisor $K_{{\Xgn}}$ of ${\Xgn}$ is Cartier. 
\end{lemma}

\begin{proof}
Let 
$(\mathcal{X}\to\mathcal{U}, \mathcal{L})$ 
be a quasi-polarized Kuranishi family, 
$\tilde{\pi}\colon \mathcal{X}_{n}\to \mathcal{U}$ 
be its $n$-fold fiber product, and 
$\bar{\mathcal{X}}_{n}\to\mathcal{U}$ 
be its relative contraction as before. 
It suffices to show that $K_{\bar{\mathcal{X}}_{n}}$ is Cartier. 
Since $K_{\tilde{\pi}}|_{F}\simeq \mathcal{O}_{F}$ 
for each fiber $F$ of $\tilde{\pi}$, 
we have 
$K_{\tilde{\pi}}\simeq \mathcal{O}_{\mathcal{X}_{n}}$ 
after shrinking $\mathcal{U}$, 
so $K_{\mathcal{X}_{n}}\simeq \mathcal{O}_{\mathcal{X}_{n}}$. 
This shows that 
$K_{(\bar{\mathcal{X}}_{n})_{reg}}\simeq \mathcal{O}_{(\bar{\mathcal{X}}_{n})_{reg}}$. 
Since $\bar{\mathcal{X}}_{n}$ is normal and 
its singular locus has codimension $\geq 3$, 
$K_{(\bar{\mathcal{X}}_{n})_{reg}}$ extends to 
an invertible sheaf isomorphic to $\mathcal{O}_{\bar{\mathcal{X}}_{n}}$. 
\end{proof}

Let 
$K_{\pi}=K_{\mathcal{X}_{g,n}^{+}}\otimes \pi^{\ast}K_{\mathcal{D}_{g}}^{-1}$ 
be the relative canonical bundle of $\pi$. 
Since $K_{\mathcal{X}_{g,n}^{+}}$ and $K_{\mathcal{D}_{g}}$ are ${\Gg}$-linearized, 
$K_{\pi}$ is also ${\Gg}$-linearized. 

\begin{proposition}\label{prop: isom of LB}
We have ${\Gg}$-equivariant isomorphisms 
\begin{equation}\label{eqn: modular Hodge}
\pi_{\ast}K_{\pi} \simeq \lambda^{\otimes n}, \quad 
K_{\pi} \simeq \pi^{\ast}\lambda^{\otimes n} 
\end{equation}
of line bundles on $\mathcal{D}_{g}$ and on $\mathcal{X}_{g,n}^{+}$. 
\end{proposition} 

\begin{proof}
We argue for $\pi\colon {\Xgn}\to {\Dg}$. 
Since 
$K_{\pi}|_{\bar{F}}\simeq K_{\bar{F}}\simeq \mathcal{O}_{\bar{F}}$ 
for each fiber $\bar{F}$ of ${\Xgn}\to{\Dg}$, 
$\pi_{\ast}K_{\pi}$ is invertible 
and the natural homomorphism 
$\pi^{\ast}\pi_{\ast}K_{\pi}\to K_{\pi}$ 
is an isomorphism. 
Hence the two isomorphisms in \eqref{eqn: modular Hodge} 
are equivalent to each other. 
Furthermore, we can use induction on $n$ to reduce the problem to the case $n=1$. 
Indeed, in the cartesian diagram 
\begin{equation*}
\xymatrix{
{\Xgn} \ar[r]^{\pi_{2}} \ar[d]_{\pi_{4}} & {\Xg} \ar[d]^{\pi_{1}} \\ 
\mathcal{X}_{g,n-1} \ar[r]_{\pi_{3}}   &    {\Dg} 
}
\end{equation*}
we have 
\begin{equation*}
K_{\pi} 
\simeq K_{\pi_{4}}\otimes \pi_{4}^{\ast}K_{\pi_{3}} 
\simeq \pi_{2}^{\ast}K_{\pi_{1}}\otimes \pi_{4}^{\ast}K_{\pi_{3}}. 
\end{equation*}

So we let $n=1$ and prove 
$\pi_{\ast}K_{\pi}\simeq \lambda$. 
Let 
$({\XgBR}\stackrel{\tilde{\pi}}{\to}{\BRg}, \varphi)$ 
be the smooth universal family over ${\BRg}$  
as in \S \ref{sssec: univ family over BRg} and 
$\bar{\pi} \colon \bar{\mathcal{X}}_{g}\to {\BRg}$ 
be its contraction.  
The projection 
$p\colon {\BRg}\to{\Dg}$ 
is the period map for the family 
$({\XgBR}\to{\BRg}, \varphi)$. 
For each $u\in{\BRg}$, 
writing $S=({\XgBR})_{u}$, 
the marking $\varphi_{u}$ sends 
$H^{2,0}(S)\subset H^{2}(S, {\C})$ 
to the line in $({\LK})_{{\C}}$ corresponding to $p(u)\in{\Dg}$.  
This means that 
$\tilde{\pi}_{\ast}K_{\tilde{\pi}}\simeq p^{\ast}\lambda$. 
Since $\bar{\pi}_{\ast}K_{\bar{\pi}}\simeq \tilde{\pi}_{\ast}K_{\tilde{\pi}}$, 
we obtain 
$\bar{\pi}_{\ast}K_{\bar{\pi}}\simeq p^{\ast}\lambda$ 
over ${\BRg}$. 
We want to show that this descends to 
$\pi_{\ast}K_{\pi} \simeq \lambda$ 
over ${\Dg}$. 
At each $u\in {\BRg}$, 
writing $\bar{S}=(\bar{\mathcal{X}}_{g})_{u}$, 
the isomorphism 
$\bar{\pi}_{\ast}K_{\bar{\pi}} \to p^{\ast}\lambda$ 
is given by 
\begin{equation}\label{eqn: Hodge-modular fiber}
H^{0}(K_{\bar{S}})\simeq H^{0}(K_{S}) \stackrel{\varphi_{u}}{\to} 
\varphi_{u}(H^{2,0}(S)) \subset ({\LK})_{{\C}}. 
\end{equation}
Even if we use another marking $w\circ \varphi_{u}$  
where $w\in W(p(u), l)$, 
this composition map does not change because 
$w$ acts trivially on the line $\varphi_{u}(H^{2,0}(S))$. 
Hence 
$\bar{\pi}_{\ast}K_{\bar{\pi}}\simeq p^{\ast}\lambda$   
descends to  
$\pi_{\ast}K_{\pi}\simeq \lambda$. 
By construction, this is ${\OLg}$-equivariant. 
\end{proof}

The Petersson metric on $\lambda^{\otimes n}$ is the $L^{2}$ metric on $\pi_{\ast}K_{\pi}$. 

\begin{proposition}\label{prop: Petersson = L2}
Let $F$ be a section of $\lambda^{\otimes n}$ over a subset $U$ of $\mathcal{D}_{g}$. 
Let $\Omega$ be the relative canonical form on $\pi^{-1}(U)\subset \mathcal{X}_{g,n}^{+}$ 
corresponding to $F$.  
Then 
\begin{equation*}
(F, F)_{n} = \int_{\pi^{-1}(U)/U}\Omega\wedge \bar{\Omega},  
\end{equation*}
where $\int_{\pi^{-1}(U)/U}$ means the fiber integral. 
\end{proposition}

\begin{proof}
This is a pointwise assertion. 
Let $[\omega]\in \mathcal{D}_{g}$ and $(S, L, \varphi)$ be a corresponding marked quasi-polarized $K3$ surface 
with the contraction $\mu\colon S \to \bar{S}$. 
We first consider the case $n=1$. 
By \eqref{eqn: Hodge-modular fiber}, 
$F([\omega])\in {\C}\omega=\varphi(H^{2,0}(S))$ 
is the image of $\mu^{\ast}\Omega_{[\omega]}\in H^{0}(K_{S})$ by $\varphi$, 
where $\Omega_{[\omega]}$ is considered as a $2$-form on $\bar{S}$. 
Since $\varphi$ is an isometry between $H^{2}(S, {\C})$ and $(\Lambda_{K3})_{{\C}}$, 
we have 
\begin{eqnarray*}
\int_{\bar{S}}\Omega_{[\omega]}\wedge \overline{\Omega_{[\omega]}} 
& = & 
\int_{S}\mu^{\ast}\Omega_{[\omega]}\wedge \overline{\mu^{\ast}\Omega_{[\omega]}} 
= (\varphi(\mu^{\ast}\Omega_{[\omega]}), \: \overline{\varphi(\mu^{\ast}\Omega_{[\omega]})}) \\ 
& = & 
(F([\omega]), \: \overline{F([\omega])}) 
= (F([\omega]), \: F([\omega]))_{1}. 
\end{eqnarray*}
For general $n$, 
the isomorphism $\pi^{\ast}\lambda^{\otimes n}\simeq K_{\pi}$ 
sends 
$\pi^{\ast}(F_{1}\otimes \cdots \otimes F_{n})([\omega])$ 
to 
$p_{1}^{\ast}\Omega_{1}\wedge \cdots \wedge p_{n}^{\ast}\Omega_{n}$ 
where 
$p_{i}\colon \bar{S}^{n}\to \bar{S}$ is the $i$-th projection 
and $\Omega_{i}\in H^{0}(K_{\bar{S}})$ corresponds to $F_{i}([\omega])\in \lambda_{[\omega]}$. 
Then our assertion follows by iterated integral. 
\end{proof}

\subsection{Proof of Theorem \ref{main thm 1}}\label{ssec: proof thm 1}

Now we prove Theorem \ref{main thm 1}. 
Let $(g, n)\ne (2, 1)$. 
First we derive the isomorphism \eqref{main eqn 1}. 
As an intermediate step, we choose a torsion-free normal subgroup 
$\Gamma \lhd {\Gg}$ of finite index and put 
$\mathcal{F}_{g,n}'=\Gamma\backslash \mathcal{X}_{g,n}^{+}$ and 
$\mathcal{F}_{g}'=\Gamma\backslash \mathcal{D}_{g}$. 
Let $\pi\colon \mathcal{F}_{g,n}'\to\mathcal{F}_{g}'$ be the projection, 
$H'\subset \mathcal{F}_{g}'$ be the image of $\mathcal{H}$, 
and $B'=\pi^{\ast}H'\subset \mathcal{F}_{g,n}'$ be the total space over $H'$. 

\begin{lemma}
For each $m, l \geq 0$ we have 
\begin{equation}\label{eqn: intermediate isom torsion-free}
H^{0}(\mathcal{F}_{g,n}', K_{\mathcal{F}_{g,n}'}^{\otimes m}(-lB')) 
\simeq M_{(19+n)m}(\Gamma,  {\rm det}^{m})^{(l)}. 
\end{equation}
\end{lemma}

\begin{proof}
By \eqref{eqn: KD} and Proposition \ref{prop: isom of LB}, 
we have ${\Gg}$-equivariant isomorphisms 
\begin{equation*}
K_{\mathcal{X}_{g,n}^{+}}^{\otimes m}(-l\pi^{\ast}\mathcal{H}) 
\simeq 
\pi^{\ast}K_{\mathcal{D}_{g}}^{\otimes m}(-l\mathcal{H}) \otimes K_{\pi}^{\otimes m}  
\simeq    
\pi^{\ast}(\lambda^{\otimes (19+n)m}\otimes {\det}^{m} (-l\mathcal{H})) 
\end{equation*}
of line bundles on $\mathcal{X}_{g,n}^{+}$. 
This descends to an isomorphism 
\begin{equation*}
K_{\mathcal{F}_{g,n}'}^{\otimes m}(-lB') 
\simeq 
\pi^{\ast} (\lambda^{\otimes (19+n)m}\otimes {\det}^{m} (-lH'))
\end{equation*}
of line bundles on $\mathcal{F}_{g,n}'$. 
Taking the global sections over $\mathcal{F}_{g,n}'$, 
we obtain \eqref{eqn: intermediate isom torsion-free}. 
\end{proof}

Let $G={\Gg}/\Gamma$. 
We take the $G$-invariant part of the isomorphism 
\eqref{eqn: intermediate isom torsion-free} with $l=m$. 
For the right hand side, we have 
\begin{equation*}
(M_{(19+n)m}(\Gamma, {\det}^{m})^{(m)})^{G} = 
M_{(19+n)m}({\Gg}, {\det}^{m})^{(m)} 
\end{equation*}
by definition. 
For the left hand side, 
let $p\colon \mathcal{F}_{g,n}'\to{\Fgn}$ be the quotient map by $G$ 
and $\mathcal{F}_{g,n}^{\circ}$ be the regular locus of ${\Fgn}$. 
Note that the singular locus of ${\Fgn}$ is of codimension $\geq 2$ by the normality of ${\Fgn}$.  
By the first statement of Proposition \ref{prop: ramification divisor} (1), 
the composition 
\begin{equation*}
\mathcal{X}_{g,n}^{+} \to \mathcal{F}_{g,n}' \stackrel{p}{\to} {\Fgn}
\end{equation*} 
is doubly ramified at $\pi^{-1}(\mathcal{H})$ 
and has no other ramification divisor. 
Since $\mathcal{X}_{g,n}^{+} \to \mathcal{F}_{g,n}'$ is unramified, 
we find that the ramification divisor of $p$ is $B'$ with ramification index $2$.  
This shows that   
$p^{\ast}K_{\mathcal{F}_{g,n}^{\circ}} \simeq K_{\mathcal{F}'_{g,n}}(-B')$ 
over $p^{-1}(\mathcal{F}_{g,n}^{\circ})\subset \mathcal{F}'_{g,n}$ 
by the Hurwitz formula. 
Therefore we obtain 
\begin{equation*}
H^{0}(\mathcal{F}_{g,n}', K_{\mathcal{F}_{g,n}'}^{\otimes m}(-mB'))^{G} 
= H^{0}(\mathcal{F}_{g,n}^{\circ}, K_{\mathcal{F}_{g,n}^{\circ}}^{\otimes m}). 
\end{equation*}
This proves the isomorphism \eqref{main eqn 1}. 
Compatibility of multiplication is evident.

Next we derive the isomorphism \eqref{main eqn 2}. 
We use the following equality. 

\begin{lemma}\label{lemma: Petersson and L2 norm}
Let $F$ be a modular form in 
$M_{19+n}({\Gg}, {\det})^{(1)}=M_{19+n}({\Gg}, {\det})$ 
and $\Omega$ be the corresponding canonical form on ${\Fgn}$. 
For any open set $U$ of ${\Fg}$ we have 
\begin{equation*}
\int_{U}(F, F)_{19+n, \det}{\volD} = \int_{\pi^{-1}(U)}\Omega \wedge \bar{\Omega}. 
\end{equation*}
\end{lemma}

\begin{proof}
Since the problem is local, we may assume that 
$U$ is a small open subset of $\mathcal{D}_{g}$, 
$F$ is a section of $\lambda^{\otimes 19+n}\otimes {\det}$ over $U$, 
and $\Omega$ is the canonical form on $\pi^{-1}(U)\subset \mathcal{X}_{g,n}^{+}$ 
corresponding to $\pi^{\ast}F$ by the isomorphism 
$K_{\mathcal{X}_{g,n}^{+}}\simeq \pi^{\ast}(\lambda^{\otimes 19+n}\otimes{\det})$. 
Since $U$ is small, we can decompose $F$ as 
$F=F_{1}\otimes F_{2}$ with 
$F_{1}$ a section of $\lambda^{\otimes 19}\otimes {\det}$ and 
$F_{2}$ a section of $\lambda^{\otimes n}$. 
Let $\Omega_{1}$ be the canonical form on $U$ corresponding to $F_{1}$, and 
$\Omega_{2}$ be the relative canonical form on $\pi^{-1}(U)$ corresponding to $F_{2}$. 
Then $\Omega=\pi^{\ast}\Omega_{1}\wedge\Omega_{2}$. 
By \eqref{eqn: volD} and Proposition \ref{prop: Petersson = L2}, 
we see that  
\begin{eqnarray*}
\int_{\pi^{-1}(U)}\Omega\wedge\bar{\Omega} 
& = & 
\int_{U} \Omega_{1}\wedge\bar{\Omega}_{1} \int_{\pi^{-1}(U)/U}\Omega_{2}\wedge\bar{\Omega}_{2} \\ 
& = & 
\int_{U}(F_{1}, F_{1})_{19,\det}(F_{2}, F_{2})_{n}{\volD} \\ 
& = & 
\int_{U}(F, F)_{19+n,\det}{\volD}. 
\end{eqnarray*}
This proves Lemma \ref{lemma: Petersson and L2 norm}.  
\end{proof}

The isomorphism \eqref{main eqn 2} now follows by comparing, 
via Lemma \ref{lemma: Petersson and L2 norm},  
Proposition \ref{prop: cusp form Petersson} with  
the well-known fact that 
a canonical form $\Omega$ on a Zariski open set $X^{\circ}$ of a smooth proper variety $X$ 
extends holomorphically over $X$ if and only if 
$\int_{X^{\circ}}\Omega\wedge\bar{\Omega}<\infty$. 
The proof of Theorem \ref{main thm 1} is completed. 
\qed 

\begin{example}
Let $g=2$. 
When $k$ is even, we have $M_{k}(\Gamma_{2}, \det)=\{ 0 \}$ 
because $-{\rm id}\in \Gamma_{2}$ acts by $-1$ on $\lambda^{\otimes k}\otimes \det$. 
Therefore, when both $n$ and $m$ are odd, 
there is no nonzero $m$-canonical form on $\mathcal{F}_{2,n}$ nor its smooth projective model. 
\end{example}

Explicitly, the correspondence \eqref{main eqn 2} is given as follows. 
We take $s_{l}, w_{0}$ as in \S \ref{ssec: tube domain}. 
Let $\Omega_{1}$ be the canonical form on $\mathcal{D}_{g}$ 
corresponding to $s_{l}^{\otimes 19}\otimes w_{0}$. 
(This is the flat canonical form $dz_{1}\wedge \cdots \wedge dz_{19}$ 
in the tube domain realization at ${\Z}l$.) 
Let $\Omega_{2}$ be the relative canonical form on $\mathcal{X}_{g,1}^{+}$ 
corresponding to $s_{l}$. 
If $F=f(s_{l}^{\otimes 19+n}\otimes w_{0})$ is a cusp form and 
$\Omega_{F}$ is the corresponding canonical form on $X$,  
the pullback of $\Omega_{F}$ by the projection 
$p\colon \mathcal{X}_{g,n}^{+} \to {\Fgn}$ is 
\begin{equation}\label{eqn: explicit correspondence}
p^{\ast}\Omega_{F} = 
(\pi^{\ast}f)\: \pi^{\ast}\Omega_{1}\wedge 
\pi_{1}^{\ast}\Omega_{2} \wedge \cdots \wedge \pi_{n}^{\ast}\Omega_{2}, 
\end{equation}
where $\pi_{i}\colon \mathcal{X}_{g,n}^{+} \to \mathcal{X}_{g,1}^{+}$ is the $i$-th projection. 
This description implies the following. 

\begin{corollary}\label{cor: canonical map}
The canonical map of $X$ factors through 
${\Fgn}\twoheadrightarrow{\Fg}\dashrightarrow {\proj}^{N}$ 
where ${\Fg}\dashrightarrow {\proj}^{N}$ is the rational map associated to 
$S_{19+n}({\Gg}, \det)$. 
\end{corollary}

\begin{proof}
If $f_{i}(s_{l}^{\otimes 19+n}\otimes w_{0})$, $0\leq i \leq N$, 
are basis of $S_{19+n}({\Gg}, \det)$, 
the canonical map of $X$ is given by 
$[\pi^{\ast}f_{0}\colon \cdots \colon \pi^{\ast}f_{N}]$. 
\end{proof}

\begin{corollary}
Let $Y$ be a smooth projective model of ${\Fgn}/\frak{S}_{n}$. 
Then \eqref{main eqn 2} also induces an isomorphism 
$S_{19+n}(\Gamma_{g}, \det )\simeq H^{0}(K_{Y})$. 
\end{corollary}

\begin{proof}
By \eqref{eqn: explicit correspondence}, $\Omega_{F}$ is $\frak{S}_{n}$-invariant 
and hence descends to a canonical form on an open set of $Y$. 
Since it has finite $L^{2}$ norm, it extends over $Y$.  
\end{proof}

When $g>2$, $Y$ can be thought of as a smooth projective model of 
the universal family of the Hilbert schemes ${\rm Hilb}^{n}(S)$.

The divisor of $\Omega_{F}$ on ${\Fgn}$ is described as follows. 
Let ${\rm div}(F)'$ be the divisor of $F$ over ${\Fg}$, 
where the vanishing order at a branch divisor of $\mathcal{D}_{g}\to {\Fg}$ 
is counted as $1/2$ of the vanishing order at the ramification divisor over it. 
Let $H$ be the image of $\mathcal{H}$ in ${\Fg}$ 
and $\pi\colon {\Fgn}\to{\Fg}$ be the projection.  

\begin{corollary}\label{cor: zero divisor}
${\rm div}(\Omega_{F})|_{{\Fgn}} = 
\pi^{\ast}( {\rm div}(F)' - H/2)$. 
\end{corollary}

\begin{proof}
This follows from \eqref{eqn: explicit correspondence} and the Hurwitz formula, 
noticing that $\pi^{\ast}H$ is reduced and 
$\pi^{\ast}H'$ has multiplicity $2$ for other branch divisors $H'$. 
\end{proof}

\begin{remark}\label{remark: dispense with}
If one is interested only in \eqref{main eqn 2}, 
one can in fact dispense with the partial compactification ${\Fgn}$, because 
$\int_{\Gamma\backslash \mathcal{D}}(F, F)_{k,\chi}{\volD}$ 
must diverge if a (meromorphic) modular form $F$ has pole at $\mathcal{H}$. 
However, in order to understand a full picture 
of the connection of modular forms and the geometry of universal family, 
such as \textit{pluri}canonical forms and Eichler-Shimura theory, 
it is necessary to fill the $(-2)$-Heegner divisor. 
\end{remark}

\subsection{Extension over compactification}\label{ssec: main thm 2}

In this subsection we assume that $X$ is a complex analytic variety 
containing ${\Fgn}$ as a Zariski open set 
and regular in codimension $1$, such that 
${\Fgn}\to{\Fg}$ extends to a morphism $\pi\colon X\to {\Fgcpt}$ 
to some toroidal compactification ${\Fgcpt}$ of ${\Fg}$ with the property that 
every irreducible component of the boundary divisor $\Delta_{X}$ of $X$ 
dominates some irreducible component of the boundary divisor $\Delta_{\mathcal{F}}$ of ${\Fgcpt}$. 
We explain how the results of Theorem \ref{main thm 1} 
could be extended for such $X$, 
clarifying the contribution from the boundary $\Delta_{X}$. 
This assumption means  
that except some codimension $\geq 2$ locus 
$X$ is an equidimensional family over ${\Fgcpt}$, 
but at present no example of such a compactification has been known, 
which makes the next assertion conditional. 

\begin{proposition}\label{prop: extension conditional}
Let $X$ be a complex analytic variety as above. 

$(1)$ The isomorphism \eqref{main eqn 1} extends to 
\begin{equation*}
\bigoplus_{m} M_{(19+n)m}({\Gg}, {\rm det}^{m})^{(m)} 
\simeq \bigoplus_{m} H^{0}(X, K_{X}^{\otimes m}(m\Delta_{X})). 
\end{equation*}
This maps 
$S_{(19+n)m}({\Gg}, {\rm det}^{m})^{(m)}$ 
into $H^{0}(X, K_{X}^{\otimes m}((m-1)\Delta_{X}))$. 

$(2)$ Assume moreover that $X$ is compact and has canonical singularities. 
Then 
\begin{equation*}
\kappa({\Fgn}) \geq 
\kappa({\Fgcpt}, \: (19+n)\lambda - \Delta_{\mathcal{F}} - H/2), 
\end{equation*}
where the right hand side is the Iitaka dimension of the ${\Q}$-divisor 
$(19+n)\lambda - \Delta_{\mathcal{F}} - H/2$ of ${\Fgcpt}$. 
\end{proposition}

\begin{proof}
Since this is similar to the case of Siegel modular forms \cite{Ma3}, 
we will just indicate the outline. 
Let $\Delta$ be a component of $\Delta_{X}$, 
and $\Delta'$ be a boundary divisor of a torus embedding 
which gives the component $\pi(\Delta)$ of $\Delta_{\mathcal{F}}$. 
We take a small neighborhood $\Delta_{x}\subset \Delta$ of a general point $x$ of $\Delta$  
and let $\Delta_{y}'\subset\Delta'$ be a small open set whose image is $\pi(\Delta_{x})$. 
Let $U_{\varepsilon}, U_{\varepsilon}'$ 
be annulus bundles around 
$\Delta_{x}, \Delta_{y}'$ 
respectively, say of radius $[\varepsilon, 1]$. 
Let $F$ be a local modular form of weight $k=(19+n)m$ and character $\chi={\rm det}^{m}$ 
on a neighborhood of $\pi(\Delta_{x})$, and 
$\Omega$ be the $m$-canonical form corresponding to $F$. 
Both (1) and (2) boil down to the assertion that 
$\Omega$ has at most pole of order $m$ along $\Delta_{x}$.  

By the $L^{2/m}$-criterion in \cite{Ma3} \S 6, 
it is sufficient to show that 
$\int_{U_{\varepsilon}}||\Omega||^{2/m} = o(\varepsilon^{-\alpha})$ 
for every $\alpha>0$. 
As before, 
we can relate the $L^{2/m}$ norm of $\Omega$ to the Petersson norm of $F$. 
Pulling back to the torus embedding, 
the problem is reduced to showing that 
\begin{equation*}
\int_{U_{\varepsilon}'}(F, F)_{k,\chi}^{1/m}{\volD} = o(\varepsilon^{-\alpha}) 
\end{equation*}
for every $\alpha>0$. 
This can be proved as in the proof of Proposition \ref{prop: cusp form Petersson}: 
\begin{eqnarray*}
\int_{U_{\varepsilon}'}(F, F)_{k,\chi}^{1/m}{\volD} 
& = & 
\int_{U_{\varepsilon}'}|f(Z)|^{2/m}({\rm Im}(Z), {\rm Im}(Z))^{n}{\rm vol}_{I} \\ 
& \sim & 
\int_{U_{\varepsilon}'}|f(Z)|^{2/m} (\log |q|)^{an} |q|^{-2} dq \wedge d\bar{q} \wedge {\rm vol}'_{I} \\ 
& \leq & 
\int_{\varepsilon}^{1} (\log r)^{an}r^{-1}dr \\ 
& = & 
o(\varepsilon^{-\alpha}). 
\end{eqnarray*}
When $F$ is a cusp form, 
the integral 
$\int_{U_{\varepsilon}'}(F, F)_{k,\chi}^{1/m}{\volD}$ 
converges and hence 
$\int_{U_{\varepsilon}}||\Omega||^{2/m}$ converges. 
This proves the second assertion in (1). 
\end{proof}

\section{Mukai models and quasi-pullback of $\Phi_{12}$}\label{sec: Mukai Borcherds}

In this section we prove Theorem \ref{main thm Mukai Borcherds}.  
In \S \ref{ssec: Phi12} 
we compute the weight $k(g)$ of the quasi-pullback $F(g)$ of the Borcherds $\Phi_{12}$ form 
for each $g\leq 22$.  
By Corollary \ref{criteria for kappa}, we have 
$\kappa({\Fgn})\geq 0$ for all $n\geq n(g)$ where $n(g)=k(g)-19$. 
In \S \ref{ssec: K3 model} we proceed from the opposite direction: 
for many $g\leq 20$, we study an upper bound $n'(g)$ where ${\Fgn}$ is unirational or uniruled 
by using classical and Mukai models of polarized $K3$ surfaces. 

With the geometric bound $n'(g)$, 
the isomorphism \eqref{main eqn 2} can also be used in the opposite direction: 
we find that $S_{k}(\Gamma_{g}, \det )=\{ 0 \}$ when $k\leq n'(g)+19$. 
The closeness of $n(g)$ and $n'(g)$ suggests that 
$F(g)$ would be the cusp form of minimal weight, at least for the character $\det$.

\subsection{Quasi-pullback of $\Phi_{12}$}\label{ssec: Phi12} 

Let $II_{2,26}=2U\oplus 3E_{8}$ be the even unimodular lattice of signature $(2, 26)$ 
and $\mathcal{D}_{2,26}$ be the associated Hermitian symmetric domain.  
In \cite{Bo}, Borcherds discovered a modular form $\Phi_{12}$ 
on $\mathcal{D}_{2,26}$ of weight $12$ and character $\det$ for ${\rm O}^{+}(II_{2,26})$, 
whose zero divisor is exactly the $(-2)$-Heegner divisor of $\mathcal{D}_{2,26}$. 
The quasi-pullback of $\Phi_{12}$ to $\mathcal{D}_{g}$ is defined as follows (\cite{Bo}, \cite{BKPSB}). 
We choose a primitive vector $v_{g}$ of norm $2-2g$ from $E_{8}$ 
and put $K_{g}=v_{g}^{\perp}\cap E_{8}$. 
(This is in general not unique even up to ${\rm O}(E_{8})$.) 
This defines a primitive embedding 
${\Lg}\hookrightarrow II_{2,26}$ 
with $\Lambda_{g}^{\perp}=K_{g}$, 
and thus an embedding 
$\mathcal{D}_{g}\hookrightarrow \mathcal{D}_{2,26}$ 
of domain. 
We divide $\Phi_{12}$ by all zeros containing $\mathcal{D}_{g}$ 
and restrict it to $\mathcal{D}_{g}$. 
More precisely, we put 
\begin{equation*}
F(g) = \left. \frac{\Phi_{12}}{\prod_{\delta}(\delta, \cdot)} \right|_{\mathcal{D}_{g}}
\end{equation*}
where $\delta$ ranges over all $(-2)$-vectors in $K_{g}$ up to $\pm 1$, 
and $(\delta, \cdot)$ is the section of $\mathcal{O}(1)$ defined by paring with $\delta$. 
The basic properties of $F(g)$ are summarized as follows. 

\begin{proposition}[\cite{BKPSB}, \cite{Ko}, \cite{GHS1}]\label{prop: quasi-pullback Phi12}
Let $r(g)$ be the number of $(-2)$-vectors in $K_{g}$. 
Then $F(g)$ is a modular form on $\mathcal{D}_{g}$ 
of weight $k(g)=12+r(g)/2$ and character $\det$ with respect to ${\Gg}$,  
and vanishes at the $(-2)$-Heegner divisor of $\mathcal{D}_{g}$. 
When $r(g)>0$, $F(g)$ is a cusp form. 
\end{proposition}

As noticed in Corollary \ref{cor: vanish at H}, 
vanishing at the $(-2)$-Heegner divisor is in fact automatic. 
Combining the results so far, we see the following. 

\begin{proposition}
Assume that $r(g)>0$ for our choice of $v_{g}\in E_{8}$. 
Then a smooth projective model of $\mathcal{F}_{g,n(g)}$ has positive geometric genus  
where $n(g)=r(g)/2-7$. 
In particular, $\kappa({\Fgn})\geq 0$ for all $n\geq n(g)$. 
Furthermore, if we have a compactification of ${\Fgn}$ for some $n>n(g)$ 
as in Proposition \ref{prop: extension conditional} (2), 
then $\kappa(\mathcal{F}_{g,n'})=19$ for all $n'\geq n$. 
\end{proposition}

\begin{proof}
By Proposition \ref{prop: quasi-pullback Phi12}, 
$F(g)$ is a nonzero element of $S_{k(g)}({\Gg}, \det)$. 
By Theorem \ref{main thm 1}, this corresponds to 
a nonzero canonical form on a smooth projective model of $\mathcal{F}_{g,n(g)}$ 
where $n(g)=k(g)-19=r(g)/2-7$. 
As for the last assertion, when $n>n(g)$, the ${\Q}$-divisor 
\begin{equation*}
(19+n)\lambda - H/2 -\Delta_{\mathcal{F}} + \det 
 =  
(n-n(g))\lambda + (k(g)\lambda - H/2 - \Delta_{\mathcal{F}} +\det) 
\end{equation*} 
of ${\Fgcpt}$ is big + effective = big.  
Then $\kappa({\Fgn})=19$ by Proposition \ref{prop: extension conditional} (2). 
\end{proof}

Let $\Omega_{F(g)}$ be the canonical form corresponding to $F(g)$. 
Its divisor on $\mathcal{F}_{g,n(g)}$ can be explicitly calculated as follows. 
For simplicity we assume $g\not\equiv 2$ mod $4$. 
Let $A=\Lambda_{g}^{\vee}/\Lambda_{g}$ be the discriminant group of $\Lambda_{g}$. 
For $\lambda \in A$ and $x\in {\Q}_{<0}$, let  
$\mathcal{H}(\lambda, x)=\bigcup_{l} (l^{\perp}\cap \mathcal{D}_{g})$ 
be the reduced Heegner divisor of discriminant $(\lambda, x)$ on $\mathcal{D}_{g}$, 
where $l$ run over all vectors in 
$\Lambda_{g}+\lambda\subset \Lambda_{g}^{\vee}$ 
with $(l, l)=2x$. 
In particular, $\mathcal{H}=\mathcal{H}(0, -1)$. 
Let $H(\lambda, x)$ be the (reduced) image of $\mathcal{H}(\lambda, x)$ in ${\Fg}$.  
On the other hand, for $y\in {\Q}_{<0}$, 
let $c_{\lambda}(y)$ be the number of vectors $v$ in 
$K_{g}+\lambda\subset K_{g}^{\vee}$ with $(v, v)=2y$, 
where we identify $A\simeq K_{g}^{\vee}/K_{g}$ as abelian groups naturally (\cite{Ni1}).  

\begin{corollary}
Let $g\not\equiv 2$ mod $4$. 
We have 
\begin{equation*}
{\rm div}(\Omega_{F(g)})|_{\mathcal{F}_{g,n(g)}} = 
\sum_{\begin{subarray}{c} \lambda \in A/\pm 1 \\ (\lambda, \lambda)\not\in 2{\Z} \end{subarray}}
c_{\lambda}(-1-x(\lambda))\pi^{-1}(H(\lambda, x(\lambda))),  
\end{equation*}
where $x(\lambda)$ is the rational number with $-1< x(\lambda) <0$ and $x(\lambda)\equiv (\lambda, \lambda)/2$ mod ${\Z}$. 
In particular, $\Omega_{F(g)}$ does not vanish at a general point of $\pi^{-1}(H)$. 
\end{corollary}

\begin{proof}
The divisor of $F(g)$ on $\mathcal{D}_{g}$ is 
\begin{equation*}
\mathcal{H} + \sum_{\begin{subarray}{c} \lambda \in A/\pm 1 \\ \lambda \ne 0 \end{subarray}}
\sum_{\begin{subarray}{c} x\equiv (\lambda, \lambda)/2 \\ 0< x <1 \end{subarray}}
c_{\lambda}(-x)\mathcal{H}(\lambda, x-1). 
\end{equation*}
This can be seen directly as in \cite{BKPSB} 
or using the fact that $F(g)$ is the Borcherds lift of 
$\Theta_{K_{g}(-1)}/\Delta$ 
(see \cite{GHS1} Remark 1). 
Since $\mathcal{H}$ has no common component with $\mathcal{H}(\lambda, x-1)$ when $g\not\equiv 2$ mod $4$, 
our assertion follows from Corollary \ref{cor: zero divisor}. 
\end{proof}

In what follows, we explicitly take a vector $v_{g}\in E_{8}$  
for each $g\leq 22$ and calculate $n(g)=r(g)/2-7$. 
The result is summarized in Theorem \ref{main thm Mukai Borcherds}. 

Recall (\cite{CS}) that the $E_{8}$-lattice is defined as 
\begin{equation}\label{eqn: E8}
E_{8} = 
\{ \: (x_{i})\in {\Q}^{8} \: | \: 
\forall x_{i}\in {\Z} \: \textrm{or} \: \forall x_{i}\in {\Z}+1/2, \; 
x_{1}+ \cdots +x_{8}\in 2{\Z} \: \},  
\end{equation}
where we take the $(-1)$-scaling of the standard quadratic form on ${\Q}^{8}$. 
The roots of $E_{8}$ are as follows. 
For $j\ne k$ we define  
$\delta_{\pm j, \pm k}=(x_{i})\in E_{8}$ by 
$x_{j}=\pm 1$, $x_{k}=\pm 1$ and $x_{i}=0$ for $i\ne j, k$. 
For a subset $S\subset \{ 1, \cdots, 8 \}$ of even cardinality we define 
$\delta'_{S}=(x_{i})\in E_{8}$ by  
$x_{i}=1/2$ if $i\in S$ and $x_{i}=-1/2$ if $i\not\in S$. 
These are the $112+128=240$ roots of $E_{8}$. 
When $g\leq 8$, $g\ne 6$ and when $g=11, 16$, 
the lattice $K_{g}$ will be obtained by deleting one vertex from 
the Dynkin diagram of $E_{8}$. 
We number the vertices of $E_{8}$ as follows. 
\begin{equation*}\label{Dynkin diagram}
\xygraph{
    \bullet ([]!{+(0,-.3)} {\delta_1}) - [r]
    \bullet ([]!{+(0,-.3)} {\delta_2}) - [r]
    \bullet ([]!{+(0,-.3)} {\delta_3}) - [r]
    \bullet ([]!{+(0,-.3)} {\delta_4}) - [r]
    \bullet ([]!{+(.3,-.3)} {\delta_5}) (
        - [u] \bullet ([]!{+(.3,0)} {\delta_8}),
        - [r] \bullet ([]!{+(0,-.3)} {\delta_6})
        - [r] \bullet ([]!{+(0,-.3)} {\delta_7})
)}
\end{equation*}
The number $r(\Lambda)$ of roots
in other ADE root lattice $\Lambda$ is given by (\cite{CS}) 
\begin{equation*}
r(A_{N}) = N(N+1), \quad r(D_{N})=2N(N-1), \quad r(E_{7})=126, \quad r(E_{6})=72. 
\end{equation*}

\vspace{2mm} 

\underline{$g=2$}
 
\noindent
The roots $\delta_{2}, \cdots, \delta_{8}$ generate $E_{7}\subset E_{8}$ 
with $E_{7}^{\perp}=\langle -2 \rangle$. 
Thus $K_{2}=E_{7}$ and $r(2)=r(E_{7})=126$. 
Hence $n(2)=56$. 
 
\vspace{1mm} 

\underline{$g=3$}

\noindent
The roots $\delta_{1}, \cdots, \delta_{6}, \delta_{8}$ generate $D_{7}\subset E_{8}$ 
with $D_{7}^{\perp}=\langle -4 \rangle$. 
Thus $K_{3}=D_{7}$ and $r(3)=r(D_{7})=84$. 
Hence $n(3)=35$. 

\vspace{1mm} 

\underline{$g=4$}

\noindent
The roots $\delta_{1}, \delta_{3}, \cdots, \delta_{8}$ generate 
$A_{1}\oplus E_{6}\subset E_{8}$ 
with $(A_{1}\oplus E_{6})^{\perp}=\langle -6 \rangle$. 
Thus $K_{4}=A_{1}\oplus E_{6}$ and $r(4)=r(A_{1})+r(E_{6})=74$. 
Hence $n(4)=30$. 

\vspace{1mm} 

\underline{$g=5$}

\noindent
The roots $\delta_{1}, \cdots, \delta_{7}$ generate 
$A_{7}\subset E_{8}$ 
with $A_{7}^{\perp}=\langle -8 \rangle$. 
Thus $K_{5}=A_{7}$ and $r(5)=r(A_{7})=56$. 
Hence $n(5)=21$. 

\vspace{1mm} 

\underline{$g=6$}

\noindent
We take a norm $-10$ vector $v_{6}\in E_{8}$ by 
$(3, 1, 0, \cdots, 0)$ in the presentation \eqref{eqn: E8}. 
The roots of $E_{8}$ orthogonal to $v_{6}$ are 
$\delta_{\pm i, \pm j}$ with $i, j\geq 3$. 
Then 
$r(6)=2^{2}\cdot \dbinom{6}{2}=60$ 
and $n(6)=23$. 
These roots form a root system of type $D_{6}$. 

\vspace{1mm}

\underline{$g=7$}

\noindent
The roots $\delta_{1}, \delta_{2}, \delta_{4}, \cdots, \delta_{8}$ generate 
$A_{2}\oplus D_{5}\subset E_{8}$ 
with $(A_{2}\oplus D_{5})^{\perp}=\langle -12 \rangle$. 
Thus $K_{7}=A_{2}\oplus D_{5}$ and $r(7)=r(A_{2})+r(D_{5})=46$. 
Hence $n(7)=16$. 

\vspace{1mm}

\underline{$g=8$}

\noindent
The roots $\delta_{1}, \cdots, \delta_{5}, \delta_{7}, \delta_{8}$ generate 
$A_{6}\oplus A_{1}\subset E_{8}$ 
with $(A_{6}\oplus A_{1})^{\perp}=\langle -14 \rangle$. 
Thus $K_{8}=A_{6}\oplus A_{1}$ and $r(8)=r(A_{6})+r(A_{1})=44$. 
Hence $n(8)=15$. 

\vspace{1mm} 

\underline{$g=9$}

\noindent
We take $v_{9}=(3, 1, \cdots, 1)$. 
The roots orthogonal to $v_{9}$ are 
$\delta_{i, -j}$ with $i, j\geq 2$. 
Hence 
$r(9)= 7\cdot 6  = 42$ 
and $n(9)=14$. 
These roots form a root system of type $A_{6}$. 

\vspace{1mm} 

\underline{$g=10$}

\noindent
We take  
$v_{10}=(4, 1, 1, 0, \cdots, 0)$. 
The roots orthogonal to $v_{10}$ are 
$\delta_{\pm i,\pm j}$ with $i, j\geq 4$ 
and $\pm \delta_{2,-3}$. 
Hence $r(10)= 2^{2}\cdot \dbinom{5}{2} + 2 = 42$ 
and $n(10)=14$. 
These roots form a root system of type $D_{5}+A_{1}$.

\vspace{1mm} 

\underline{$g=11$}

\noindent
The roots $\delta_{1}, \delta_{2}, \delta_{3}, \delta_{5}, \cdots, \delta_{8}$ generate 
$A_{3}\oplus A_{4}\subset E_{8}$ 
with $(A_{3}\oplus A_{4})^{\perp}=\langle -20 \rangle$. 
Thus $K_{11}=A_{3}\oplus A_{4}$ and $r(11)=r(A_{3})+r(A_{4})=32$. 
Hence $n(11)=9$. 
 
\vspace{1mm} 

\underline{$g=12$}

\noindent
We take  
$v_{12}=(4, 1, \cdots, 1, 0)$. 
The roots orthogonal to $v_{12}$ are 
$\delta_{i,-j}$ with $2\leq i, j\leq 7$ 
and $\pm\delta'_{1,i}$ with $2\leq i \leq 7$. 
Hence $r(12)= 6\cdot 5 + 2\cdot 6 = 42$ 
and $n(12)=14$.

\vspace{1mm} 

\underline{$g=13$}

\noindent
We take 
$v_{13}=(4, 2, 1, 1, 1, 1, 0, 0)$. 
The roots orthogonal to $v_{13}$ are 
$\delta_{\pm7,\pm8}$, 
$\delta_{i,-j}$ with $3\leq i, j\leq 6$, 
$\pm \delta'_{1,i}$ with $3\leq i \leq 6$, and  
$\pm \delta'_{1,i,7,8}$ with $3\leq i \leq 6$. 
Hence $r(13)= 4+12+8+8=32$ 
and $n(13)=9$.

\vspace{1mm} 

\underline{$g=14$}

\noindent
We take 
$v_{14}=(3, 3, 2, 1, 1, 1, 1, 0)$. 
As before we can calculate $r(14)=32$ 
and $n(14)=9$.

\vspace{1mm} 

\underline{$g=15$}

\noindent
We take 
$v_{15}=(5, 1, 1, 1, 0, 0, 0, 0)$. 
Then $r(15)=30$ and $n(15)=8$.

\vspace{1mm} 

\underline{$g=16$}

\noindent
The roots $\delta_{1}, \cdots, \delta_{4}, \delta_{6}, \delta_{7}, \delta_{8}$ generate 
$A_{4}\oplus A_{2} \oplus A_{1}\subset E_{8}$ 
with $(A_{4}\oplus A_{2} \oplus A_{1})^{\perp}=\langle -30 \rangle$. 
Thus $K_{16}=A_{4}\oplus A_{2} \oplus A_{1}$ and 
$r(16)=r(A_{4})+r(A_{2})+r(A_{1})=28$. 
Hence $n(16)=7$. 

\vspace{1mm} 

\underline{$g=17$}

\noindent
We take 
$v_{17}=(5, 2, 1, 1, 1, 0, 0, 0)$. 
Then $r(17)=26$ 
and $n(17)=6$.

\vspace{1mm} 

\underline{$g=18$}

\noindent
We take 
$v_{18}=(5, 2, 1, \cdots, 1, 0)$. 
Then $r(18)=30$ 
and $n(18)=8$.

\vspace{1mm}

\underline{$g=19$}

\noindent
We take 
$v_{19}=(5, 2, 2, 1, 1, 1, 0, 0)$. 
Then $r(19)=24$ 
and $n(19)=5$.

\vspace{1mm} 

\underline{$g=20$}

\noindent
We take 
$v_{20}=(5, 2, 2, 2, 1, 0, 0, 0)$. 
The roots orthogonal to $v_{20}$ are 
$\delta_{i,-j}$ with $2\leq i, j\leq4$, 
$\delta_{\pm i, \pm j}$ with $i, j\geq 6$, 
$\pm \delta_{1,5}'$, and 
$\pm \delta_{1,5,i,j}'$ with $i, j\geq6$. 
Hence 
$r(20)=6+12+2+6=26$ 
and $n(20)=6$. 

\vspace{1mm} 

\underline{$g=21$}

\noindent
We take 
$v_{21}=(6, 1, 1, 1, 1, 0, 0, 0)$. 
Then $r(21)=24$ 
and $n(21)=5$. 

\vspace{1mm} 

\underline{$g=22$}

\noindent
We take 
$v_{22}=(6, 2, 1, 1, 0, 0, 0, 0)$. 
Then $r(22)=26$ 
and $n(22)=6$.

\subsection{Unirationality/Uniruledness}\label{ssec: K3 model}

Next we study a bound $n'(g)$ where ${\Fgn}$ is unirational or uniruled. 
We recall known results in $8\leq g \leq 11$, $g=14, 22$, 
and compute this bound for other $g\leq 20$ using classical and Mukai models of 
polarized $K3$ surfaces of genus $g$. 
The result is summarized in Theorem \ref{main thm Mukai Borcherds}. 
In a few cases, even rationality holds.

\subsubsection{The case $g=2$}

General $K3$ surfaces of degree $2$ are surfaces of weighted degree $6$ in 
$Y={\proj}(1, 1, 1, 3)$, parametrized by an open set of $|\mathcal{O}_{Y}(6)|$.  

\begin{proposition}
$\mathcal{F}_{2,38}$ is unirational. 
\end{proposition}

\begin{proof}
Consider the (complete) incidence correspondence 
\begin{equation*}
X_{n} = 
\{ \: (S, p_{1}, \cdots, p_{n}) \in |\mathcal{O}_{Y}(6)| \times Y^{n} \: | \: p_{i}\in S \: \}  
\; \subset \; |\mathcal{O}_{Y}(6)| \times Y^{n}  
\end{equation*}
and let $\pi\colon X_{n}\to Y^{n}$ be the projection. 
Then $X_{n}$ is irreducible and 
the $\pi$-fiber over $(p_{1}, \cdots, p_{n})\in Y^{n}$ is 
the linear system of surfaces in $|\mathcal{O}_{Y}(6)|$ 
passing through $p_{1}, \cdots, p_{n}$. 
This has dimension $\geq \dim |\mathcal{O}_{Y}(6)| - n = 38-n$, 
hence is non-empty for general $(p_{1}, \cdots, p_{n})$ when $n\leq 38$. 
Thus, if $n\leq 38$, 
$X_{n}$ is birationally a projective space bundle over $Y^{n}$ and so rational.   
Since $\mathcal{F}_{2,n}$ is dominated by $X_{n}$, 
it is unirational for $n\leq 38$. 
\end{proof}

\subsubsection{The case $g=3$}

$K3$ surfaces of degree $4$ are quartics in ${\proj}^{3}$. 

\begin{proposition}
$\mathcal{F}_{3,34}$ is rational. 
\end{proposition}

\begin{proof}
Consider the (complete) incidence correspondence 
\begin{equation*}
X_{n} = 
\{ \: (S, p_{1}, \cdots, p_{n}) \in |\mathcal{O}_{{\proj}^{3}}(4)| \times ({\proj}^{3})^{n} 
\: | \: p_{i}\in S \: \}  
\; \subset \; |\mathcal{O}_{{\proj}^{3}}(4)| \times ({\proj}^{3})^{n},  
\end{equation*} 
which is irreducible. 
The fiber of the projection 
$X_{n}\to ({\proj}^{3})^{n}$ over $(p_{1}, \cdots, p_{n})\in ({\proj}^{3})^{n}$ 
is the linear system of quartics through $p_{1}, \cdots, p_{n}$. 
Thus, when $n\leq \dim |\mathcal{O}_{{\proj}^{3}}(4)| = 34$, 
$X_{n}$ is birationally a ${\proj}^{34-n}$-bundle over $({\proj}^{3})^{n}$. 
In particular, $X_{34}\sim ({\proj}^{3})^{34}$. 
Therefore 
$\mathcal{F}_{3,34}\sim X_{34}/{\rm PGL}_{4}\simeq ({\proj}^{3})^{29}$. 
\end{proof}

More generally, when $3\leq n \leq 33$, 
$\mathcal{F}_{3,n}$ is rational too. 

\subsubsection{The case $g=4$}

General $K3$ surfaces of degree $6$ are $(2, 3)$ complete intersections in ${\proj}^{4}$. 
It is convenient to fix a smooth quadric $Q\subset {\proj}^{4}$ 
and consider surfaces in $|\mathcal{O}_{Q}(3)|=|\!-\! K_{Q}|$. 

\begin{proposition}
$\mathcal{F}_{4,29}$ is unirational. 
\end{proposition}

\begin{proof}
Consider the (complete) incidence correspondence 
\begin{equation*}
X_{n} = 
\{ \: (S, p_{1}, \cdots, p_{n}) \in |\mathcal{O}_{Q}(3)| \times Q^{n} 
\: | \: p_{i}\in S \: \}  
\; \subset \; |\mathcal{O}_{Q}(3)| \times Q^{n}.  
\end{equation*}  
The fiber of the projection  
$X_{n}\to Q^{n}$ over $(p_{1}, \cdots, p_{n})\in Q^{n}$ 
is the linear system of surfaces in $|\mathcal{O}_{Q}(3)|$ 
passing through $p_{1}, \cdots, p_{n}$. 
Thus, when $n\leq \dim |\mathcal{O}_{Q}(3)|=29$, 
$X_{n}\to Q^{n}$ is birationally a ${\proj}^{29-n}$-bundle and hence $X_{n}$ is rational. 
Therefore $\mathcal{F}_{4,n}$ is unirational. 
\end{proof}

\subsubsection{The case $g=5$}

General $K3$ surfaces of degree $8$ are $(2, 2, 2)$ complete intersections in ${\proj}^{5}$. 
They are parametrized by an open subset $G(3, V)^{\circ}$ 
of the Grassmannian $G(3, V)$ where $V=H^{0}(\mathcal{O}_{{\proj}^{5}}(2))$. 
For $W\in G(3, V)^{\circ}$ we denote by $S_{W}\subset {\proj}^{5}$ 
the $K3$ surface defined by $W$. 

\begin{proposition}
$\mathcal{F}_{5,18}$ is rational. 
\end{proposition}

\begin{proof}
Consider the incidence correspondence    
\begin{equation*}
X_{n} = 
\{ \: (W, p_{1}, \cdots, p_{n}) \in G(3, V)^{\circ} \times ({\proj}^{5})^{n} 
\: | \: p_{i}\in S_{W} \: \}  
\; \subset \; G(3, V)^{\circ} \times ({\proj}^{5})^{n}.  
\end{equation*}  
The fiber of the projection $X_{n}\to ({\proj}^{5})^{n}$ 
over general $(p_{1}, \cdots, p_{n})\in ({\proj}^{5})^{n}$ 
is an open dense subset of the sub Grassmannian $G(3, V_{p_{1},\cdots,p_{n}})$,
where $V_{p_{1},\cdots,p_{n}}\subset V$ is 
the space of quadratic forms vanishing at $p_{1},\cdots,p_{n}$. 
When $n\leq h^{0}(\mathcal{O}_{{\proj}^{5}}(2))-3=18$, we have $\dim V_{p_{1},\cdots,p_{n}} \geq 3$, 
so $X_{n}$ is birationally a $G(3, 21-n)$-bundle over $({\proj}^{5})^{n}$. 
Therefore 
\begin{equation*}
\mathcal{F}_{5,18} \sim 
X_{18}/{\rm PGL}_{6} \sim 
({\proj}^{5})^{18}/{\rm PGL}_{6} \sim 
({\proj}^{5})^{11}
\end{equation*}
is rational.  
\end{proof}

More generally, when $7\leq n \leq 17$, 
$\mathcal{F}_{5,n}$ is rational too. 

\subsubsection{The case $g=6$}

In \cite{Mu1} \S 4, Mukai proved that 
general $K3$ surfaces of genus $6$ are 
anti-canonical sections of the Fano $3$-fold $Y$ of index $2$ and degree $5$. 
The variety $Y$ is unique up to isomorphism and is quasi-homogeneous for ${\rm PGL}(2)$. 
Since $h^{0}(-K_{Y})=23$ and $Y$ is rational, 
the same argument as before implies 

\begin{proposition}
$\mathcal{F}_{6,22}$ is unirational. 
\end{proposition}

\subsubsection{The case $g=7$}

Mukai \cite{Mu1}, \cite{Mu5} found two methods of 
constructing general $K3$ surfaces of genus $7$. 
Farkas-Verra \cite{FV} proved that $\mathcal{F}_{7,n}$ is unirational in $n\leq 8$ 
using the first model (\cite{Mu1}). 
Here we use the second model (\cite{Mu5} \S 3). 
Let $G(2, 5)\subset {\proj}^{9}$ be the Pl\"ucker embedding of the Grassmannian $G(2, 5)$, 
and $Y=G(2, 5)\cap P$ be the intersection with 
a general codimension $2$ linear subspace $P\subset {\proj}^{9}$. 
The variety $Y$ is a quintic del Pezzo $4$-fold (unique up to isomorphism), 
and the stabilizer $G\subset {\rm PGL}_{5}$ of $Y$ has dimension $8$. 
Let $\mathcal{E}$ be the dual of the universal sub bundle over $G(2, 5)$, 
and put $\mathcal{E}_{Y}=\mathcal{E}\otimes \mathcal{O}(1)|_{Y}$. 
By \cite{Mu5}, we have $h^{0}(\mathcal{E}_{Y})=30$, 
and the zero locus of a general section of $\mathcal{E}_{Y}$ 
is a $K3$ surface of genus $7$. 
This gives a dominant map 
${\proj}H^{0}(\mathcal{E}_{Y})/G \dashrightarrow \mathcal{F}_{7}$ 
whose general fibers are birationally $K3$ surfaces 
(Fourier-Mukai partner of the original $K3$ surface). 
This construction tells us 

\begin{proposition}
$\mathcal{F}_{7,14}$ is uniruled. 
\end{proposition}

\begin{proof}
Let $U\subset {\proj}H^{0}(\mathcal{E}_{Y})$ be the open locus of 
sections $\tau$ whose zero locus $S_{\tau}$ is a $K3$ surface. 
Consider the incidence correspondence 
\begin{equation*}
X_{14} = 
\{ \: ([\tau], p_{1}, \cdots, p_{14}) \in U \times Y^{14} 
\: | \: p_{i}\in S_{\tau} \: \}  
\; \subset \; U \times Y^{14},  
\end{equation*} 
which is irreducible. 
Let $\pi\colon X_{14}\to Y^{14}$ be the projection. 
If $\mathbf{p}=(p_{1}, \cdots, p_{14})\in \pi(X_{14})$, 
the fiber $\pi^{-1}(\mathbf{p})$ is a non-empty open set of 
the linear subspace ${\proj}V_{\mathbf{p}}\subset {\proj}H^{0}(\mathcal{E}_{Y})$ 
of sections vanishing at $p_{1}, \cdots, p_{14}$. 
Since $\dim V_{\mathbf{p}}\geq h^{0}(\mathcal{E}_{Y})-2\cdot 14 =2$, 
then $X_{14}$ is birationally a ${\proj}^{m}$-bundle over $\pi(X_{14})$ with $m\geq 1$. 
Since a general $\mathbf{p}\in \pi(X_{14})$ has trivial stabilizer in $G$ (cf.~\cite{PZ}), 
$X_{14}/G$ is uniruled. 
We have a dominant map 
$X_{14}/G\dashrightarrow \mathcal{F}_{7,14}$ 
whose general fibers are birationally $K3$ surfaces. 
Thus $\mathcal{F}_{7,14}$ is also uniruled.  
\end{proof}

The proof shows that 
the fibration of Fourier-Mukai partners over $\mathcal{F}_{7,14}$ 
(dual $K3$ fibration) is uniruled. 
This space is akin to $\mathcal{F}_{7,15}$, which is the only missing case. 

\begin{remark}
The $D_{5}$ root system appears in two contexts for $g=7$: 
as a sub root system of $K_{7}=D_{5}\oplus A_{2}$, 
and as the Dynkin diagram of ${\rm SO}(10)$ whose spinor variety is 
used in the first Mukai model 
(cf.~\cite{Mu3} (1.11)). 
\end{remark}

\subsubsection{The case $8\leq g \leq10$}\label{sssec:g=8,9,10}

Mukai \cite{Mu1} proved that 
general $K3$ surfaces of genus $7\leq g \leq 10$ are 
linear sections of a homogeneous space $\Sigma_{g}$ 
embedded in ${\proj}V_{g}$ for a representation space $V_{g}$. 
Using this fact, Farkas-Verra \cite{FV} proved that 
${\Fgn}$ is unirational in $n\leq g+1$.

\subsubsection{Interlude: a coincidence}\label{mystery}
As we have seen from the classical and Mukai models \cite{Mu1}, 
general $K3$ surfaces of genus $3\leq g \leq 10$ are linear sections of 
a (quasi-)homogeneous space embedded in ${\proj}V_{g}$ 
for a representation $V_{g}$ of an algebraic group $G_{g}$. 
Then we find that the following coincidence always holds: 
\begin{equation*}
n(g) = \dim V_{g}, 
\end{equation*}
namely, 
the weight $k(g)$ of $F(g)$ minus $19$ (the moduli number) coincides with $\dim V_{g}$. 
Is this accidental? 
Is there some deeper link between $K3$ surfaces and the Borcherds product $F(g)$ behind this coincidence? 

We summarize $(G, V)=(G_{g}, V_{g})$ in the following table.  
\begin{center}
\begin{tabular}{cccccccccccc} \toprule  
$g$       & 3   & 4   & 5   & 6   & 7 & 8 & 9 & 10    
\\ \midrule  
$G$       & ${\rm SL}(4)$   & ${\rm SO}(5)$   & ${\rm SL}(6)$   & ${\rm SL}(2)$  
 & ${\rm SO}(10)$ & ${\rm SL}(6)$ & ${\rm Sp}(6)$ & $G_{2}$    
\\ \midrule  
$V$  & $S^{4}{\C}^{4}$ & $ S^{3}{\C}^{5}/{\C}^{5}$ & 
$S^{2}{\C}^{6}$ & $V_{6}$ & 
$S^{+}$ & $\bigwedge^{2}{\C}^{6}$ & $\bigwedge^{3}{\C}^{6}/{\C}^{6} $ & $\frak{g}_{2}$ 
\\ \bottomrule \\ 
\end{tabular}
\end{center}
Here $V_{6}=S^{0}{\C}^{2}\oplus S^{8}{\C}^{2} \oplus S^{12}{\C}^{2}$ and 
$S^{+}$ is the half spin representation. 

In $g=3, 4, 6, 12$, 
$n(g)$ also coincides with $h^{0}(-K_{Y})$ 
for Fano $3$-folds $Y$ containing the $K3$ surfaces. 
This would be a uniform reason why $n(g)=n'(g)+1$ is attained at these $g$.

\subsubsection{The case $g=11$}

Barros \cite{Ba} proved that $\mathcal{F}_{11,n}$ is unirational in $n\leq 6$ and 
uniruled in $n=7$.  
Barros-Mullane \cite{BM} proved that $\kappa(\mathcal{F}_{11,9})=0$. 
Thus the cusp form $F(11)$ gives the unique nonzero canonical form 
on a smooth projective model of  $\mathcal{F}_{11,9}$. 
Farkas-Verra \cite{FV} proved that $\kappa(\mathcal{F}_{11,11})=19$.

\subsubsection{The case $g=12$}

Mukai (\cite{Mu3}, \cite{Mu5}) proved that 
general $K3$ surfaces of genus $12$ are 
anti-canonical sections of Fano $3$-folds of genus $12$ and Picard number $1$. 
We first recall the construction of such Fano $3$-folds following \cite{Mu3} \S 5. 
Let $Y=G(3, 7)$ and $\mathcal{E}$ be the dual of the universal sub bundle over $Y$. 
We set $V=H^{0}(\bigwedge^{2}\mathcal{E}) \simeq \bigwedge^{2}({\C}^{7})^{\vee}$. 
For a $3$-dimensional subspace $N$ of $V$, 
we let $Y_{N}\subset Y$ be the common zero locus of sections in $N$. 
If $N$ is general, then $Y_{N}$ is a Fano $3$-fold of genus $12$ and Picard number $1$. 
Such Fano $3$-folds are rational (\cite{IP} Theorem 4.6.7). 
Restriction of the Pl\"ucker embedding 
$Y\hookrightarrow {\proj}(\bigwedge^{3}{\C}^{7})$ 
gives the anti-canonical embedding 
$Y_{N}\hookrightarrow {\proj}N^{\perp}\simeq {\proj}^{13}$ 
where $N^{\perp}\subset \bigwedge^{3}{\C}^{7}$ 
is the annihilator of $N\subset \bigwedge^{2}({\C}^{7})^{\vee}$. 
If $H$ is a general hyperplane of ${\proj}N^{\perp}$, 
then $S_{N,H}=Y_{N}\cap H$ is a $K3$ surface of genus $12$. 
By Mukai \cite{Mu5}, $Y_{N}$ can be uniquely recovered from $S_{N,H}$ inside $Y$, 
so that $\mathcal{F}_{12}$ is birationally a ${\proj}^{13}$-bundle over 
the moduli space $G(3, V)/{\rm PGL}_{7}$ 
of these Fano $3$-folds with fibers the anti-canonical systems. 

\begin{proposition}
$\mathcal{F}_{12,13}$ is birational to 
the moduli space of $13$-pointed Fano $3$-folds of genus $12$ and Picard number $1$. 
In particular, $\mathcal{F}_{12,13}$ is rationally connected. 
\end{proposition}

\begin{proof}
Let $\mathcal{F}_{12,13}'$ be the moduli space of 
$13$-pointed Fano $3$-folds of genus $12$ and Picard number $1$. 
We have the rational map 
$\mathcal{F}_{12,13}'\dashrightarrow \mathcal{F}_{12,13}$ 
which sends $(Y_{N}, p_{1}, \cdots, p_{13})$ to $(S_{N,H}, p_{1}, \cdots, p_{13})$ 
where $H=\langle p_{1}, \cdots, p_{13} \rangle$ is the hyperplane of ${\proj}N^{\perp}$ 
spanned by $p_{1}, \cdots, p_{13}$. 
Since $Y_{N}$ can be recovered from $S_{N,H}$, 
we also have the inverse map 
$\mathcal{F}_{12,13}\dashrightarrow \mathcal{F}_{12,13}'$ 
sending $(S_{N,H}, p_{1}, \cdots, p_{13})$ to $(Y_{N}, p_{1}, \cdots, p_{13})$. 
Therefore $\mathcal{F}_{12,13}$ is birational to $\mathcal{F}_{12,13}'$. 

Since $\mathcal{F}_{12,13}'$ is fibered over the unirational base 
$G(3, V)/{\rm PGL}_{7}$ with general fibers rational 
(self products of the Fano $3$-folds), 
it is rationally connected by \cite{GrHaSt}. 
\end{proof}

\subsubsection{The case $g=13$}

By Mukai \cite{Mu5}, general $K3$ surfaces of genus $13$ can be constructed as follows. 
Consider the Grassmannian $G(3, 7)$ and let $\mathcal{E}$ be the dual of the universal sub bundle. 
We take a general $2$-dimensional subspace of 
$H^{0}(\bigwedge^{2}\mathcal{E})\simeq \bigwedge^{2}({\C}^{7})^{\vee}$, 
which is unique up to the action of ${\rm SL}_{7}$, 
and let $Y\subset G(3, 7)$ be its common zero locus. 
Then $Y$ is a Fano $6$-fold of index $3$, 
and its stabilizer $G_{Y}$ in ${\rm SL}_{7}$ has dimension $10$. 
Let $\mathcal{F}$ be 
the third exterior power of the universal quotient bundle over $G(3, 7)$, 
and $\mathcal{F}_{Y}=\mathcal{F}|_{Y}$ 
be its restriction to $Y$. 
Then $\mathcal{F}_{Y}$ has rank $4$, 
$h^{0}(\mathcal{F}_{Y})=32$, 
and the zero locus of a general section of $\mathcal{F}_{Y}$ is a $K3$ surface of genus $13$. 
The resulting moduli map 
${\proj}H^{0}(\mathcal{F}_{Y})/G_{Y}\dashrightarrow \mathcal{F}_{13}$ 
is dominant and its general fibers are birationally $K3$ surfaces 
(Fourier-Mukai partner of the original $K3$ surface). 

\begin{proposition}
$\mathcal{F}_{13,7}$ is uniruled. 
\end{proposition}

\begin{proof}
Let $U\subset {\proj}H^{0}(\mathcal{F}_{Y})$ 
be the open set of sections $\tau$ whose zero locus $S_{\tau}$ is a $K3$ surface. 
Consider the incidence correspondence 
\begin{equation*}
X_{n} = 
\{ \: ([\tau], p_{1}, \cdots, p_{n}) \in U\times Y^{n}\: | \: 
p_{i}\in S_{\tau} \: \},  
\end{equation*}
which is irreducible, 
and let $\pi\colon X_{n}\to Y^{n}$ be the projection. 
The $\pi$-fiber over 
$\mathbf{p}=(p_{1}, \cdots, p_{n})\in \pi(X_{n})$ 
is an open dense subset of 
the linear subspace 
${\proj}V_{\mathbf{p}}\subset {\proj}H^{0}(\mathcal{F}_{Y})$ 
of sections vanishing at $p_{1}, \cdots, p_{n}$. 
When $n=7$, $V_{\mathbf{p}}$ has dimension 
$\geq h^{0}(\mathcal{F}_{Y})-4\cdot 7=4$, 
so $X_{7}\to \pi(X_{7})$ is birationally a ${\proj}^{m}$-bundle with $m\geq 3$. 
Since the stabilizer in $G_{Y}$ of general $\mathbf{p}\in \pi(X_{7})$ is finite, 
$X_{7}/G_{Y}$ is uniruled. 
Since the fibers of the moduli map 
$X_{7}/G_{Y}\dashrightarrow \mathcal{F}_{13,7}$ 
are $K3$ surfaces, 
$\mathcal{F}_{13,7}$ is also uniruled. 
\end{proof}

The proof shows that 
the fibration of Fourier-Mukai partners over $\mathcal{F}_{13,7}$ 
(dual $K3$ fibration) is also uniruled. 
Note that $\mathcal{F}_{13,8}$ is the only missing case.

\subsubsection{The case $g=14$}

Farkas-Verra \cite{FV} proved that $\mathcal{F}_{14,1}$ is rational 
using special cubic $4$-folds.

\subsubsection{The case $g=16$}

By Mukai \cite{Mu6}, general $K3$ surfaces of genus $16$ can be constructed as follows. 
Let $\mathcal{T}$ be the ${\rm GL}_{4}$-equivariant compactification of 
the space of twisted cubics in ${\proj}^{3}$ constructed by 
Ellingsrud-Piene-Stromme \cite{EPS}. 
The variety $\mathcal{T}$ is smooth of dimension $12$ and acted on by ${\rm GL}_{4}$. 
There exist vector bundles $\mathcal{E}$, $\mathcal{F}$ on $\mathcal{T}$ 
of rank $3$, $2$ respectively and equivariant linear maps 
$(S^{2}{\C}^{4})^{\vee}\to H^{0}(\mathcal{E})$ and  
$(S^{2,1}{\C}^{4})^{\vee}\to H^{0}(\mathcal{F})$. 
Here $S^{2}{\C}^{4}$ is the symmetric square of ${\C}^{4}$ and 
$S^{2,1}{\C}^{4}$ is the kernel of ${\C}^{4}\otimes S^{2}{\C}^{4}\to S^{3}{\C}^{4}$. 
We write $V=(S^{2}{\C}^{4})^{\vee}$. 
If $M\subset V$ is a general $2$-dimensional subspace, 
its common zero locus $Y_{M}$ has dimension $6$. 
Let ${\it Syz}_{M}\subset S^{2,1}{\C}^{4}$ be the kernel of 
${\C}^{4}\otimes M^{\perp}\to S^{3}{\C}^{4}$, 
which has dimension $12$. 
Then a general $2$-dimensional subspace $N$ of ${\it Syz}_{M}^{\vee}$ 
cuts out a $K3$ surface $S_{M,N}\subset Y_{M}$ of genus $16$. 

The subspaces $M$ are parametrized by an open set 
$G(2, V)^{\circ}$ of $G(2, V)$, 
and the pairs $(M, N)$ are parametrized by an open set $\mathcal{P}^{\circ}$ 
of the $G(2, 12)$-bundle $\mathcal{P}$ over $G(2, V)^{\circ}$ 
formed by $G(2, {\it Syz}_{M}^{\vee})$. 
By Theorem 1.2 and Remark 7.3 of \cite{Mu6}, 
the moduli map 
$\mathcal{P}/{\rm PGL}_{4}\dashrightarrow \mathcal{F}_{16}$ 
is dominant and its general fibers are birationally $K3$ surfaces 
(Fourier-Mukai partner of the original $K3$ surface). 

\begin{proposition}
$\mathcal{F}_{16,4}$ is uniruled. 
\end{proposition}

\begin{proof}
As in the case $g=12$, we consider the incidence correspondence in two steps: 
\begin{equation*}
Z_{n} = 
\{ \: (M, p_{1}, \cdots, p_{n}) \in G(2, V)^{\circ}\times \mathcal{T}^{n} \: | \: p_{i}\in Y_{M} \: \}, 
\end{equation*} 
\begin{equation*}
X_{n} = 
\{ \: (M, N, p_{1}, \cdots, p_{n}) \in \mathcal{P}^{\circ}\times \mathcal{T}^{n} 
\: | \: p_{i}\in S_{M, N} \: \}.  
\end{equation*} 
Let $\pi\colon X_{n}\to Z_{n}$ be the projection. 
If $(M, \mathbf{p})=(M, p_{1}, \cdots, p_{n})\in \pi(X_{n})$, 
the $\pi$-fiber over $(M, \mathbf{p})$ is an open dense subset of the sub Grassmannian 
$G(2, W_{M,\mathbf{p}})\subset G(2, {\it Syz}_{M}^{\vee})$ 
where 
$W_{M,\mathbf{p}}\subset {\it Syz}_{M}^{\vee}$ 
is the kernel of the evaluation map at $p_{1}, \cdots, p_{n}$. 
If $n\leq 4$, $W_{M,\mathbf{p}}$ has dimension 
$\geq \dim {\it Syz}_{M}^{\vee}-2n\geq 4$, 
so $X_{n}\to \pi(X_{n})$ is birationally a $G(2, m)$-bundle with $m\geq 4$. 
When $n=4$, 
a general $(M, \mathbf{p})\in \pi(X_{4})$ has finite stabilizer in ${\rm PGL}_{4}$, 
and hence $X_{4}/{\rm PGL}_{4}$ is uniruled. 
Since the moduli map 
$X_{4}/{\rm PGL}_{4}\dashrightarrow \mathcal{F}_{16,4}$ 
is dominant and its general fibers are birationally $K3$ surfaces, 
$\mathcal{F}_{16,4}$ is also uniruled. 
\end{proof}

As in the cases $g=7, 13$, the proof shows that 
the dual $K3$ fibration over $\mathcal{F}_{16,4}$ is also uniruled.

\subsubsection{The case $g=18$}

By Mukai \cite{Mu2}, 
general $K3$ surfaces of genus $18$ can be constructed as follows. 
Fix a smooth quadric $Q$ in ${\proj}^{8}$ and 
let $Y=OG(3, 9)$ be the orthogonal Grassmannian 
parametrizing $2$-planes contained in $Q$. 
We have a rank $2$ homogeneous vector bundle $\mathcal{F}$ over $Y$ 
such that $V=H^{0}(Y, \mathcal{F})$ is the 
$16$-dimensional spin representation of ${\rm Spin}(9)$. 
If $N$ is a general $5$-dimensional subspace of $V$, 
its common zero locus $S_{N}$ is a $K3$ surface of genus $18$. 
The moduli map 
$G(5, V)/{\rm SO}(9) \dashrightarrow \mathcal{F}_{18}$ 
is birational. 

\begin{proposition}
$\mathcal{F}_{18,5}$ is uniruled. 
\end{proposition}

\begin{proof}
Let $G(5, V)^{\circ}$ be the open set of $G(5, V)$ 
where the common zero locus $S_{N}$ is a $K3$ surface. 
Consider the incidence correspondence 
\begin{equation*}
X_{n} = 
\{ \: (N, p_{1}, \cdots, p_{n}) \in G(5, V)^{\circ} \times Y^{n} 
\: | \: p_{i}\in S_{N} \: \},  
\end{equation*} 
and let $\pi\colon X_{n}\to Y^{n}$ be the projection. 
For $\mathbf{p}=(p_{1}, \cdots, p_{n})\in \pi(X_{n})$,  
the fiber $\pi^{-1}(\mathbf{p})$ is an open dense subset of 
the sub Grassmannian $G(5, V_{\mathbf{p}})$ 
where $V_{\mathbf{p}}\subset V$ is the subspace of sections vanishing at $p_{1}, \cdots, p_{n}$. 
When $n=5$, $V_{\mathbf{p}}$ has dimension 
$\geq \dim V - 5 \cdot {\rm rk}(\mathcal{F}) = 6$, 
so $\pi^{-1}(\mathbf{p})$ has positive dimension. 
Since the stabilizer in ${\rm SO}(9)$ of general $\mathbf{p}\in \pi(X_{5})$  
has dimension $< \dim \pi^{-1}(\mathbf{p})$,  
the fiber $\pi^{-1}(\mathbf{p})$ is not contracted to one point by the moduli map 
$X_{5}\dashrightarrow X_{5}/{\rm SO}(9) \sim \mathcal{F}_{18,5}$. 
Therefore $\mathcal{F}_{18,5}$ is uniruled.  
\end{proof}

\subsubsection{The case $g=20$}

Let $\mathcal{E}$ be the dual of the universal sub bundle over 
the Grassmannian $Y=G(4, 9)$,  
and put $V=H^{0}(\bigwedge^{2}\mathcal{E})\simeq \bigwedge^{2}({\C}^{9})^{\vee}$. 
By Mukai \cite{Mu2}, if $N$ is a general $3$-dimensional subspace of $V$, 
its common zero locus $S_{N}$ is a $K3$ surface of genus $20$. 
The moduli map 
$G(3, V)/{\rm PGL}_{9}\dashrightarrow \mathcal{F}_{20}$ 
is birational. 

\begin{proposition}
$\mathcal{F}_{20,5}$ is uniruled. 
\end{proposition}

\begin{proof}
Let $G(3, V)^{\circ}$ be the open set of $G(3, V)$ 
where the common zero locus $S_{N}$ is a $K3$ surface. 
Consider the incidence correspondence 
\begin{equation*}
X_{n} = \{ \: (N, p_{1}, \cdots, p_{n})\in 
G(3, V)^{\circ} \times Y^{n} \: | \: p_{i}\in S_{N} \: \},  
\end{equation*}
and let 
$\pi\colon X_{n}\to Y^{n}$ be the projection. 
If $\mathbf{p}=(p_{1}, \cdots, p_{n})\in \pi(X_{n})$, 
then $\pi^{-1}(\mathbf{p})$ is an open dense subset of the sub Grassmannian 
$G(3, V_{\mathbf{p}})\subset G(3, V)$ 
where 
$V_{\mathbf{p}}\subset V$ is the subspace of sections vanishing at $p_{1}, \cdots, p_{n}$. 
When $n=5$,  
$V_{\mathbf{p}}$ has dimension  
$\geq \dim V - 5\cdot {\rm rk}(\bigwedge^{2}\mathcal{E}) = 6$, 
so $\pi^{-1}(\mathbf{p})$ has positive dimension. 
Since the stabilizer in ${\rm PGL}_{9}$ of general $(p_{1}, \cdots, p_{5})\in \pi(X_{5})$ is finite, 
$\pi^{-1}(\mathbf{p})$ is not contracted by the moduli map 
$X_{5}\dashrightarrow X_{5}/{\rm PGL}_{9} \sim \mathcal{F}_{20,5}$. 
Therefore 
$\mathcal{F}_{20,5}$ is uniruled. 
\end{proof}

\subsubsection{The case $g=22$}

Farkas-Verra \cite{FV2} proved that $\mathcal{F}_{22,1}$ is unirational 
using special cubic $4$-folds.


\end{document}